\documentclass[12pt,reqno]{amsart}
\usepackage[T1]{fontenc}
\usepackage[utf8]{inputenc}
\usepackage{lmodern}
\usepackage{microtype}
\usepackage{amsmath,amssymb,mathtools}
\usepackage{enumitem,hyperref}
\hypersetup{
    colorlinks=true,
    linkcolor=red,
    citecolor=green,
    urlcolor=magenta
}
\usepackage{needspace}
\usepackage[margin=1in]{geometry}
\usepackage{url}
\usepackage{color}

\definecolor{blue}{rgb}{0,0,1}
\definecolor{red}{rgb}{1,0,0}
\definecolor{green}{rgb}{0,.6,.2}
\definecolor{purple}{rgb}{1,0,1}

\long\def\red#1\endred{\textcolor{red}{#1}}
\long\def\blue#1\endblue{\textcolor{blue}{#1}}
\long\def\purple#1\endpurple{\textcolor{purple}{ #1}}
\long\def\green#1\endgreen{\textcolor{green}{#1}}

\numberwithin{equation}{section}

\newtheorem{theorem}{Theorem}[section]
\newtheorem{proposition}[theorem]{Proposition}
\newtheorem{lemma}[theorem]{Lemma}
\newtheorem{corollary}[theorem]{Corollary}
\theoremstyle{definition}
\newtheorem{definition}[theorem]{Definition}

\theoremstyle{remark}
\newtheorem*{remark}{Remark}

\newcommand{\Hh}{\mathbb H}
\newcommand{\C}{\mathbb C}
\newcommand{\R}{\mathbb R}
\newcommand{\PSL}{\mathrm{PSL}_2(\mathbb Z)}

\newcommand{\covol}{\operatorname{covolume}}
\newcommand{\Span}{\operatorname{span}}
\newcommand{\ran}{\operatorname{ran}}
\makeatletter
\def\ps@headings{\ps@empty
  \def\@evenhead{%
    \setTrue{runhead}%
    \normalfont\scriptsize
    \rlap{\thepage}\hfil
    \def\thanks{\protect\thanks@warning}%
    \leftmark{}{}\hfil}%
  \def\@oddhead{%
    \setTrue{runhead}%
    \normalfont\scriptsize\hfil
    \def\thanks{\protect\thanks@warning}%
    {\fontsize{8}{10}\selectfont\rightmark{}{}}\hfil\llap{\thepage}}%
  \let\@mkboth\markboth
}
\makeatother
\title[Tracelike vectors, cusp forms, and von Neumann algebras]{On a question of Jones on tracelike vectors, cusp forms, and von Neumann algebras}
\author{Nikolaos Diamantis}
\address{School of Mathematical Sciences,
University of Nottingham
Nottingham, NG9 1HB
United Kingdom}
\email{nikolaos.diamantis@nottingham.ac.uk}
\author{Larry Rolen}
\address{Department of Mathematics, 1420 Stevenson Center, Vanderbilt University, Nashville, TN 37240}
\email{larry.rolen@vanderbilt.edu}
\date{}

\newcommand{\Tr}{\operatorname{Tr}}
\newcommand{\HH}{\mathbb H}

\newcommand{\cH}{\mathcal H}
\newcommand{\cF}{\mathcal F}
\newcommand{\That}{\widehat T}
\newcommand{\diag}{\operatorname{diag}}

\newcommand{\ip}[2]{\langle #1,#2\rangle}
\newcommand{\norm}[1]{\lVert #1\rVert}
\newcommand{\Prob}{\mathbb P}

\allowdisplaybreaks[1]
\hypersetup{pdftitle={On a question of Jones on tracelike vectors, cusp forms, and von Neumann algebras},pdfauthor={Nikolaos Diamantis and Larry Rolen}}
\begin{document}
\begin{abstract}
Using the theory of von Neumann algebras, Vaughan Jones showed abstractly that there is a bridge between holomorphic cusp forms and random matrices. This bridge is supplied by a distinguished function he called a {\it tracelike vector}. He posed the following problem: construct a tracelike vector to make this bridge explicit. We carry out Jones' vision
for the full modular group. We construct holomorphic tracelike vectors for every finite-covolume Fuchsian group throughout the permissible Bergman range. For the modular group at the critical parameter $s=13$, we use the resulting unit tracelike vector to assign explicit group coefficients to cusp form multiplication operators, and form positive finite random matrices from a shared Haar-distributed unitary matrix. Almost surely, their normalised traces recover all joint operator moments, first as the matrix size tends to infinity, and then as the group-word cutoff tends to infinity. 
We also identify the limiting mixed second moment with a series of Petersson norms that involve Maass-Shimura operators of cusp forms.
\end{abstract}
\maketitle

\section{Introduction}
Vaughan Jones established a connection between cusp forms and random matrices whose explicit realisation required a holomorphic ``tracelike'' vector \cite{Jones}. As Jones pointed out, this connection has potential important arithmetic applications, but only if such an explicit realisation is provided. 

For the modular group $\Gamma=\PSL$, the connection can be briefly outlined as follows. R\u{a}dulescu's work \cite{Radulescu} identifies cusp form operators as a dense source of operators in the commutant of the group action on weighted Bergman space. At the critical parameter $s=13$, Jones's dimension calculation gives an abstract identification of this commutant with the group von Neumann algebra. Voiculescu's asymptotic-freeness theorem \cite[Theorem 3.8]{Voiculescu} supplies random-matrix models for that algebra. Jones singled out the missing explicit tracelike vector in \cite[Introduction and Problem~1]{Jones} and, for the modular group, in the discussion entitled ``The dream'' in \cite[slides~38 and~41]{JonesSem}. 

The purpose of this paper is to provide an explicit realisation of the above connection, as asked by Jones. 
Specifically, we will construct explicit random matrices from a given set of cusp forms such that all joint trace moments converge to a deterministic limit. Our main result can be outlined as follows.
\begin{theorem}\label{jointIntro}
Let $f_1,\ldots,f_m$ be a fixed finite collection of nonzero holomorphic cusp forms for $\Gamma$, with weights allowed to differ. There are explicitly specified positive $N\times N$ matrix-valued random variables $Z_{f_j,N,L}$, for $6\mid N$ and $L\ge0$, such that on an event of probability one, for every $r\ge1$ and every $i_1,\ldots,i_r\in\{1,\ldots,m\}$, we have
\[
 \lim_{L\to\infty}\lim_{\substack{N\to\infty\\6\mid N}}
 \frac1N\Tr\bigl(Z_{f_{i_1},N,L}\cdots Z_{f_{i_r},N,L}\bigr)
 =\tau'_{13}\bigl(\That_{f_{i_1}}\cdots\That_{f_{i_r}}\bigr)
\]
for some operators $\That_{f_i}$ and a trace $\tau'_{13}$ explicitly defined in Section \ref{sec:coefficients} and \eqref{tau's} respectively.
\end{theorem}
The precise theorem is stated in Theorem \ref{thm:joint}, but, as with the following result, we have presented it here informally to make the overall concept clearer.

In the case when $r=2$, the limit has a simpler interpretation in terms of modular forms. Specifically, in Theorem~\ref{thm:raised} we evaluate the limit as a convergent positive series of Petersson norms of products involving the classical Maass-Shimura operators. 
\begin{theorem} 
For nonzero $f\in S_k(\Gamma)$ and $g\in S_\ell(\Gamma)$, we have
\begin{equation*}
\tau'_{13}(\That_f\That_g)=\frac{4\pi}{(\ell+12)N_fN_g}
 \sum_{n=0}^\infty\frac{(4\pi)^{2n}}{n!(\ell+13)_n}
 \norm{f\delta_\ell^{(n)}g}_{\text{Pet}}^2,
\end{equation*}
where $\delta_\ell^{(n)}$ are iterated Maass-Shimura operators (see Section \ref{sec:raised}), $N_f$ is a normalising constant \eqref{eq:mass}, and $||\cdot||_{\text{Pet}}$ denotes the Petersson norm in weight $k+\ell+2n$.
\end{theorem}

As mentioned above, a key element of the explicit construction is the tracelike vector. In this work, we present it for a broader scope than that required for the connection between cusp forms and random matrices. Specifically, consider any Fuchsian group of finite hyperbolic covolume $\Gamma<\mathrm{PSL}_2(\R)$. Set
\[
 s_0:=1+\frac{4\pi}{\covol(\Gamma)}.
\]
For every $1<s\le s_0$, Theorem~\ref{thm:tracelike} constructs a nonzero holomorphic vector $\xi_s$ in a suitable space, satisfying the defining equation of a tracelike vector, namely
\[
 \sum_{\gamma\in\Gamma}|\pi_s(\gamma)\xi_s(z)|^2
 =\frac{s-1}{4\pi}y^{-s}.
\]
An independent paper by Abreu~\cite{Abreu} employs a closely related polar normalisation of a kernel orbit to obtain a tracelike vector in the case $\Gamma=\mathrm{PSL}_2(\mathbb Z)$ and $s=s_0=13$. 

We developed the explicit realisation of the connection between cusp forms and random matrices presented here over the course of a conversation with ChatGPT. The exchange concerned, in particular, the search for a random-matrix model realising the joint noncommutative distribution of the cusp form Toeplitz operators $\hat T_f$ (see \eqref{eq:toeplitz1}), using Voiculescu's asymptotic-freeness framework \cite{Voiculescu}. It also concerned the choice of truncation and limiting procedure in Theorem \ref{jointIntro}: first letting the matrix size $N$ tend to infinity at fixed word cutoff $L$, and then letting $L$ tend to infinity. An key guide for our search was R\u{a}dulescu's result on the linear span of the operators $\hat T_f$ in the commutant for $\Gamma=\mathrm{PSL}_2(\mathbb Z)$ (\cite[Remark 9.6]{Jones}). 

The exchange with ChatGPT that led to the construction of the explicit tracelike vector was based on the authors’ nonholomorphic solution of a weaker problem using an analogue of Cohen’s kernel studied in \cite{DO}. The subsequent exploration included modifications of that solution, expansions in a Bergman basis, and an infinite semidefinite/rank-one system. After unsuccessful approaches and corrections, the discussion led to the use of the orbit of a Bergman reproducing kernel and to the polar normalisation of the map $\mathcal C$ underlying our construction. The latter is a map allowing us to pass from sequences parametrised by group elements to Bergman spaces and vice-versa \eqref{eq:C-extension}.

The paper is organised as follows. In Section \ref{Parts12} we provide the background on Bergman spaces and von Neumann algebras that will enable us to define tracelike vectors and state Jones's question. More importantly, it allows us in subsection \ref{part1} to make explicit the connection between cusp forms and the group von Neumann algebra for $\PSL$ by constructing the map $\mathfrak C_{\xi}$ \eqref{eq:cusp-group-map}. In Section \ref{C} we construct an explicit tracelike vector and (in subsection \ref{sec:mbekhta}) outline a way to numerically compute it. Section \ref{sec:coefficients} describes in detail the random matrices that are associated to cusp forms. In the final section we show the main theorem on the convergence of all joint trace moments to a deterministic limit. We also show that the mixed second moment can be expressed in terms of Petersson norms.

\section*{Acknowledgments}
The authors thank Don Zagier for frequent, and highly useful, discussions on the mathematics and exposition in this paper.

\section*{Guide to notation}
Here we provide a list of notation of the most prominent objects in the
paper.

\begin{itemize}
\item Throughout, $\Gamma$ denotes a Fuchsian group of finite hyperbolic covolume in $\mathrm{PSL}_2(\R)$. We
denote the corresponding quotient by 
$ X:=\Gamma\backslash\HH$.
When noted, in later sections we will sometimes specialize to $\Gamma=\PSL$. 
 The usual invariant measure is denoted by $d\mu(z):=\frac{dx\,dy}{y^2}$, for $z=x+iy$. We will sometimes normalize this to the probability measure $d\nu:=\mu(X)^{-1}d\mu.$
\item For any such $\Gamma$, the critical parameter is $s_0:=1+4\pi/\mathrm{covolume}(\Gamma)$. 
\item The notation $A_{s-2}^2(\mathbb H)$ denotes the Bergman space, given in Definition~\ref{def:Bergman}. The notations $||\cdot||_s$ and $\langle\cdot, \cdot\rangle_s$ denote the norm and inner product on this space, respectively, as defined in Section \ref{BergmanSpaces}.
\item The function  $K_{s,w}(z)$ is the reproducing kernel defined in \eqref{eq:kernel}.
\item The notation $\pi_s(\gamma)$ denotes the action of $\gamma\in\Gamma$ on the Bergman space given in \eqref{pisdefn}
\item For a von Neumann algebra $S$, the notation $S'$ denotes the commutant of $S$ \eqref{Commutant}.
\item For any $\Gamma$, $L(\Gamma)$ is the group von Neumann algebra for $\Gamma$ defined in \eqref{eq:group-vN}.
\item The von Neumann algebra $M_s$ is the bicommutant defined in \eqref{Msdefn}.
\item The Hilbert space of weight $k$ cusp forms for $\Gamma$ is denoted by $S_k(\Gamma)$. Its inner product, the Petersson innner product, is denoted by $\langle\cdot,\cdot\rangle_{\text{Pet}}$ \eqref{Peterssinn}.
\item Given any function $f$ (not necessarily holomorphic) transforming as a modular form of weight $k$ on $\PSL$, we denote its squared Petersson norm by 
$N_f:=\|f\|^2_{\text{Pet}}=\langle f,f\rangle_{\text{Pet}}$.
Its normalized ``mass'' is defined by $v_f(z):=\frac{\pi}{3N_f}y^k|f(z)|^2$. (Section \ref{sec:mass}).
\item The usual Maass-Shimura raising operator $\delta_k$ and its iterate $\delta_k^{(n)}$ are defined in \eqref{eq:raising}.
\item For any $1<s\leq s_0$, our tracelike vector $\xi_s$ is given by $\xi_s:=V\delta_e$, where $V$ is constructed in Proposition~\ref{prop:polar} and where $\delta_e$ is the characteristic function of $\{e\}$, considered as an element of $\ell^2(\Gamma)$. When $s$ is clear from context, we simply write $\xi$ for $\xi_s$.
\item For $f\in S_k(\Gamma)$, $\mathcal M_f\colon A^2_{s-2}(\Hh)\to A^2_{k+s-2}(\Hh)$ is the multiplication by $f$ map \eqref{eq:Mf}.
 \item For $f,g \in S_k(\Gamma)$, the  Toeplitz operator $T_{f,g}^{(s)}$ is the composition $\mathcal M_f^\ast\mathcal M_g$, where $^\ast$ denotes adjoint in the Bergman space $A_{s-2}^2(\mathbb H)$. 
 \item Our random matrix models are written in terms of $\That_f:=\frac{\pi}{3N_f}\mathcal M_f^*\mathcal M_f.$
  \item Our crucial linear, positivity- and trace-preserving $*$-anti-isomorphism $\mathfrak C_\xi\colon M_{13}'\longrightarrow L(\Gamma)$, constructed using $\xi_{13}$, is defined in \eqref{eq:cusp-group-map}.
\item For an $N\times N$ matrix $A$,  the normalised trace is $
\operatorname{tr}_N(A):=\frac1N\Tr(A)$,
where $\Tr(A)$ denotes the usual matrix trace.
\item The trace $\tau_\Gamma$ on $L(\Gamma)$ is defined in \eqref{eq:group-trace}
\item The trace on $M'_s$ is denoted by $\tau'_s$ (see \eqref{tau's}, \eqref{eq:traces}).
\item The inner product on $L(\Gamma)$ is defined by 
$
 \langle X,Y\rangle_2:=\tau_\Gamma(Y^*X)$ (Proposition \ref{prop:group-standard}(iii)). 
 \item The norm on $M'_s$ \eqref{2norm} is denoted by $\|T\|_{2, \tau'_s}$ 
 \item For a pair of cusp forms $f,g$, their correlation is
$C(f,g):=\tau'_s(\That_f\That_g)$.
For $f=g$, we also define the moments:
$
 m_r(f):=\tau'_s(\That_f^{\,r}).
$
\item Our random matrix variables $Z_{f_j,N,L}$ are defined by \eqref{eq:af}, \eqref{eq:generators}, \eqref{eq:H}, and \eqref{eq:Z}. The equation \eqref{eq:H} in particular defines the $H_{f,N,L}$ that occurs in the latter statement of Theorem~\ref{thm:joint}.

\end{itemize}

\section{Background and Jones's trace-vector problem}\label{Parts12}

\subsection{Bergman spaces and tracelike vectors}\label{BergmanSpaces}
This subsection deals mainly with elements of the following space.
\begin{definition}\label{def:Bergman}
For $s>1$, the (weighted) Bergman space $A^2_{s-2}(\Hh)$ is the Hilbert space of
 holomorphic functions $h$ on $\Hh$ for which
\[
 \|h\|_s^2:=\int_{\Hh}|h(z)|^2y^{s-2}\,dx\,dy<\infty.
\]
\end{definition}
The inner product on $A^2_{s-2}(\Hh)$ is given by
\[
 \langle h_1,h_2\rangle_s
 :=\int_{\Hh}h_1(z)\overline{h_2(z)}y^{s-2}\,dx\,dy.
\]
Let $\Gamma<\mathrm{PSL}_2(\R)$ be a Fuchsian group of finite hyperbolic covolume
\[
 \covol(\Gamma):=\int_{\Gamma\backslash\Hh}\frac{dx\,dy}{y^2}.
\]
As usual, $\gamma z$ denotes the M\"obius transformation action of a matrix $\gamma\in\Gamma$ on a point $z\in\mathbb H$. We also use the critical parameter
\[
 s_0:=1+\frac{4\pi}{\covol(\Gamma)}.
\]
The following is crucial for what follows. 
\begin{theorem}[{\cite{Jones}, Theorem 1.1}]\label{thm:Jones-zero}
Suppose that $z_0\in\Hh$ has trivial stabiliser in $\Gamma$. Then there is a non-zero
 $h\in A^2_{s-2}(\Hh)$ such that $h(\gamma z_0)=0$ for all $\gamma\in\Gamma$ if and only if
\[
 s>s_0.
\]
\end{theorem}

\begin{remark}
Note that for $\Gamma=\PSL$, we have $s_0=13$. Thus, a $\Gamma$-orbit of a point with trivial stabiliser is contained in the zero set of an element of $A^2_{s-2}(\Hh)$ precisely when $s>13$.
\end{remark}

Another key object for both the context and our construction is the reproducing kernel
 $K_{s,w}$ at $w\in\Hh$, given by
\begin{equation}\label{eq:kernel}
 K_{s,w}(z):=\frac{s-1}{4\pi}
 \left(\frac{z-\overline w}{2i}\right)^{-s}.
\end{equation}
The property justifying its name is that, for any
 $h\in A^2_{s-2}$, its values $h(w)$ can be obtained by 
\begin{equation}\label{eq:reproducing}
 h(w)=\langle h,K_{s,w}\rangle_s, \qquad \text{for $w \in \mathbb H.$}
\end{equation}
This follows, for example, from \cite[Proposition~4.7 (2)]{Jones}
\Needspace{5\baselineskip}
The reproducing kernel allows us to build dense spanning sets for Bergman spaces for small $s$.
\begin{proposition}\label{prop:kernel-complete}
Let $z_0\in\Hh$ have a trivial stabiliser in $\Gamma$. We have
\begin{equation}\label{eq:kernel-complete}
 \overline{\Span}\{K_{s,\gamma z_0}:\gamma\in\Gamma\}=A^2_{s-2}(\Hh)
\end{equation}
if and only if $s\leq s_0.$
\end{proposition}

\begin{proof}
By \eqref{eq:reproducing}, an element of $A^2_{s-2}(\Hh)$ is orthogonal to every $K_{s,\gamma z_0}$ precisely when it
 vanishes at every point $\gamma z_0$. The result is therefore implied by Theorem~\ref{thm:Jones-zero}.
\end{proof}

We now consider the (projective) action $\pi_s$ of $\Gamma$ on functions in $A^2_{s-2}$. Throughout, the
 non-integral powers will be interpreted via the principal logarithm. For every $\gamma\in\Gamma$,
 choose the representative
\[
 \begin{pmatrix}a_\gamma&b_\gamma\\ c_\gamma&d_\gamma\end{pmatrix}\in\mathrm{SL}_2(\R)
\]
for which either $c_\gamma>0$, or $c_\gamma=0$ and $d_\gamma>0$. For real $s>1$, we define
 $\pi_s$ in terms of the Petersson slash action:
\begin{equation}\label{pisdefn}
 \pi_s(\gamma)h:=h|_s\gamma^{-1},
 \quad\text{where }(h|_s\gamma)(z):=(c_\gamma z+d_\gamma)^{-s}h(\gamma z).
\end{equation}
It is a classical fact (see, e.g.,\ \cite[Section~2.6]{Iwaniec}) that $\pi_s(\gamma)$ is unitary and satisfies
\begin{equation}\label{eq:projective-law}
 \pi_s(\alpha)\pi_s(\beta)=\sigma_s(\alpha,\beta)\pi_s(\alpha\beta)
 \quad\text{for some }\sigma_s(\alpha,\beta) \in \mathbb T.
\end{equation}
Here, $\mathbb T:=\{z \in \mathbb C; |z|=1\}.$ 
Following Jones, we define the tracelike vectors needed below.

\begin{definition}\label{def:tracelike}
The vector $0\neq\xi\in A^2_{s-2}(\Hh)$ is tracelike if for some constant $C_\xi>0$ we have
\begin{equation}\label{eq:tracelike}
 \sum_{\gamma\in\Gamma}|\pi_s(\gamma)\xi(z)|^2=C_\xi y^{-s}.
\end{equation}
\end{definition}
With this terminology, Jones's question can be stated as follows:\\
\emph{Find an explicit tracelike vector for $\Gamma=\PSL$ and $s=s_0=13$.}

\medskip

Using standard arguments on reproducing kernels, we obtain the following characterisation of tracelike vectors.

\Needspace{5\baselineskip}
\begin{proposition}\label{prop:tracelike-frame}
A non-zero vector $\xi\in A^2_{s-2}(\Hh)$ is tracelike with constant $C_\xi$ if and only if
\begin{equation}\label{eq:tight-frame}
 \sum_{\gamma\in\Gamma}\bigl|\langle h,\pi_s(\gamma)\xi\rangle_s\bigr|^2
 =\frac{4\pi C_\xi}{s-1}\cdot\|h\|_s^2
 \qquad\text{for all }h\in A^2_{s-2}(\Hh).
\end{equation}
Moreover, if $0\neq\xi$ is tracelike, we have
\begin{equation}\label{eq:tracelike-norm}
 \|\xi\|_s^2=C_\xi\cdot\covol(\Gamma)
\end{equation}
When $s=s_0$, the tracelike vector can be normalised so that its $\Gamma$-orbit forms an orthonormal
 basis of $A^2_{s-2}(\Hh)$.
\end{proposition}

\begin{proof}
By \eqref{eq:kernel} and \eqref{eq:reproducing}, we have
\begin{equation}\label{eq:kernel-diagonal}
 \|K_{s,z}\|_s^2=K_{s,z}(z)=\frac{s-1}{4\pi}y^{-s}.
\end{equation}
If \eqref{eq:tight-frame} holds, then setting $h=K_{s,z}$ gives \eqref{eq:tracelike}.
 Conversely, if  \eqref{eq:tracelike} holds, then by \eqref{eq:kernel-diagonal}, 
\begin{equation}\label{eq:diagonal-series}
 \sum_{\gamma\in\Gamma}|\pi_s(\gamma)\xi(z)|^2
 =\frac{4\pi C_\xi}{s-1}K_{s,z}(z).
\end{equation}
This and an application of the Cauchy--Schwarz inequality imply that
\[
 F(z,w):=\sum_{\gamma\in\Gamma}\pi_s(\gamma)\xi(z)
 \overline{\pi_s(\gamma)\xi(w)}
\]
converges locally uniformly. Hence, $F(z,w)$ is holomorphic in $z$ and anti-holomorphic in $w$.
 Because of this, $F(z,w)$ is uniquely determined by its restriction to the diagonal $w=z$.
 On the other hand, by \eqref{eq:diagonal-series}, this restriction is $F(z,z)=\frac{4\pi C_\xi}{s-1}\cdot K_{s,z}(z)$.
 Hence,
\[
 F(z,w)=\frac{4\pi C_\xi}{s-1}K_{s,w}(z).
\]
Choosing $h=\sum_{j=1}^N c_jK_{s,w_j}$, it follows that
\[
 \sum_{\gamma\in\Gamma}|\langle h,\pi_s(\gamma)\xi\rangle_s|^2
 =\frac{4\pi C_\xi}{s-1}\sum_{j,k}c_j\overline{c_k}K_{s,w_j}(w_k)
 =\frac{4\pi C_\xi}{s-1}\|h\|_s^2.
\]
As in the proof of Proposition \ref{prop:kernel-complete} we see that, by the reproducing property, finite linear combinations of the kernels $K_{s,w}$, $w\in\Hh$ are dense. Because of this and since for any two such combinations $h,h_1$ we have
\[
\sum_{\gamma\in\Gamma}
|\langle h-h_1,\pi_s(\gamma)\xi\rangle_s|^2
=\frac{4\pi C_\xi}{s-1}\|h-h_1\|s^2,
\] we deduce that \eqref{eq:tight-frame} holds for all $h \in A^2_{s-2}(\Hh)$.
 
Now for any non-zero tracelike vector, multiplying \eqref{eq:tracelike} with $y^s$ and integrating over a fundamental domain in terms of
 the $\Gamma$-invariant measure $dx\,dy/y^2$, we unfold to deduce \eqref{eq:tracelike-norm}.
 
 Finally, when $s=s_0$, we can normalise the tracelike vector so that $\|\xi\|_{s_0}^2=1$.
 Since the identity term already has size one, all inner products not corresponding to the identity vanish. The projective law changes the other off-diagonal orbit inner products only by a constant, so they vanish as well. Completeness follows from \eqref{eq:tight-frame}.
\end{proof}

\subsection{von Neumann algebras}
We begin by recalling the definition of von Neumann algebras.

\begin{definition}\label{def:von-Neumann}
Let $\mathcal K$ be a Hilbert space and let $B(\mathcal K)$ be its algebra of bounded operators.
 A \emph{von Neumann algebra} on $\mathcal K$ is a subalgebra of $B(\mathcal K)$ which contains
 the identity operator, is closed under adjoints $*$, and is closed in the strong operator topology
 (that is, $T_i\to T$ strongly when $\|T_i h-Th\|\to0$ for every $h\in\mathcal K$).
\end{definition}

For $\mathcal S\subseteq B(\mathcal K)$, set
\begin{equation}\label{Commutant}
 \mathcal S':=\{T\in B(\mathcal K):TS=ST\text{ for every }S\in\mathcal S\},
 \qquad \mathcal S'':=(\mathcal S')'.
\end{equation}
We call $\mathcal S'$ the commutant of $\mathcal S$ and $\mathcal S''$ its bicommutant.
If $\mathcal S$ is closed under adjoints, von Neumann's bicommutant theorem (see, e.g. \cite[Section 2.1]{AP}) identifies
 $\mathcal S''$ with the strong topology closure of the $*$-algebra with identity operator generated by
 $\mathcal S$. We call this closure the von Neumann algebra generated by $\mathcal S$.

We are particularly interested in a natural von Neumann algebra associated with any Fuchsian
 group $\Gamma$. For such a $\Gamma$, consider the Hilbert space
\[
 \ell^2(\Gamma):=
 \left\{\mathbf a=(a_\gamma)_{\gamma\in\Gamma}:
 \sum_{\gamma\in\Gamma}|a_\gamma|^2<\infty\right\},
 \quad\text{with }\langle\mathbf a,\mathbf b\rangle_{\ell^2(\Gamma)}
 :=\sum_{\gamma\in\Gamma}a_\gamma\overline{b_\gamma}.
\]
For $\gamma\in\Gamma$, let $\delta_\gamma$ be the characteristic function of $\gamma$.
 Then $\{\delta_\gamma\}_{\gamma\in\Gamma}$ is the standard orthonormal basis of $\ell^2(\Gamma)$.
 Define the left and right regular representations of $\Gamma$ on $\ell^2(\Gamma)$ by
\[
 \lambda_\alpha\delta_\gamma:=\delta_{\alpha\gamma}
 \text{ and }\rho_\alpha\delta_\gamma:=\delta_{\gamma\alpha^{-1}}
 \text{ respectively.}
\]
We define the \emph{group von Neumann algebra} for $\Gamma$ as
\begin{equation}\label{eq:group-vN}
 L(\Gamma):=\{\lambda_\gamma:\gamma\in\Gamma\}''.
\end{equation}
The following facts can be found in, e.g., \cite[Subsection 1.3.1]{AP}.

\Needspace{5\baselineskip}
\begin{proposition}\label{prop:group-standard}
Let $e$ be the identity of $\Gamma$.
\begin{enumerate}[label=(\roman*),leftmargin=*]
\item The algebra $L(\Gamma)$ carries the trace
\begin{equation}\label{eq:group-trace}
 \tau_\Gamma(X):=\langle X\delta_e,\delta_e\rangle
\end{equation}
which is faithful and satisfies $\tau_\Gamma(I)=1$. For a finite sum
 $X=\sum_{\gamma\in \Gamma}c_\gamma\lambda_\gamma$, we have $$\tau_\Gamma(X)=c_e.$$
\item We have
\begin{equation}\label{eq:right-commutant}
 L(\Gamma)'=\{\rho_\gamma:\gamma\in\Gamma\}''.
\end{equation}
Thus, the map
\begin{equation}\label{eq:J0}
 J_0\left(\sum_\gamma a_\gamma\delta_\gamma\right)
 :=\sum_\gamma\overline{a_\gamma}\,\delta_{\gamma^{-1}}
\end{equation}
satisfies $J_0L(\Gamma)J_0=L(\Gamma)'$ and therefore, it induces an anti-isomorphism
$$
 L(\Gamma)'\longrightarrow L(\Gamma)$$ given by $$A\longmapsto J_0A^*J_0.$$
\item Define $L^2(L(\Gamma),\tau_\Gamma)$ to be the completion of $L(\Gamma)$ with respect to
\[
 \langle X,Y\rangle_2:=\tau_\Gamma(Y^*X).
\]
 Then $
 \lambda_\gamma\mapsto\delta_\gamma$ extends to a unitary operator
\[
 L^2(L(\Gamma),\tau_\Gamma)\cong\ell^2(\Gamma).
\]
Under this identification, $
1$ corresponds to $\delta_e$ and
 $J_0
 X=
 {X^*}$.
\end{enumerate}
\end{proposition}
\subsection{The connection between cusp forms and the group von Neumann algebra of $\PSL$. }\label{part1}
With the notation of the last two subsections, set
\begin{equation}\label{Msdefn}
 M_s:=\{\pi_s(\gamma):\gamma\in\Gamma\}''.
\end{equation}
Then we have
\begin{equation}\label{eq:bergman-commutant}
 M_s'=\{T\in B(A^2_{s-2}(\Hh)):
 T\,\pi_s(\gamma)=\pi_s(\gamma)\,T\text{ for every }\gamma\in\Gamma\}.
\end{equation}

In this subsection we focus on the cusp form/random-matrix connection and hence specialise to $\Gamma=\PSL$. 
 In this case, the projective action $\pi_s$ can be normalised to a standard action. Fix the generators
\[
 S=\begin{pmatrix}0&-1\\1&0\end{pmatrix},\qquad
 U=\begin{pmatrix}0&-1\\1&1\end{pmatrix}
 \quad\text{in }\PSL.
\]
We have $\Gamma=\langle S,U\mid S^2=U^3=e\rangle\cong C_2*C_3$. Thus, by \eqref{eq:projective-law}, $\pi_s(S)^2=\alpha I$ and $\pi_s(U)^3=\beta I$ for some $\alpha,\beta \in \mathbb T$.
 Choose $\omega_2,\omega_3\in\mathbb T$ with $\omega_2^2=\alpha^{-1}$ and
 $\omega_3^3=\beta^{-1}$. Then the assignment
\[
 \widetilde\pi_s(S):=\omega_2 \pi_s(S),\qquad
 \widetilde\pi_s(U):=\omega_3 \pi_s(U)
\]
extends uniquely to a unitary representation $\widetilde\pi_s$ of $\Gamma$.
 Each $\widetilde\pi_s(\gamma)$ differs from $\pi_s(\gamma)$ by an element of $\mathbb T$, so the generated
 algebra $M_s$ and all absolute-value and orthogonality identities are unchanged.
 (See \cite[Notes~4.9(1) and Appendix~A]{Jones}.) In view of this, for simplicity, we will
 denote $\widetilde\pi_s$ by $\pi_s$ throughout this subsection. With this notational convention,
\begin{equation}
    \label{action}
 \pi_s(g)\pi_s(h)=\pi_s(gh), \qquad g, h \in \PSL.
\end{equation}
As explained in \cite[Corollary~3.6 and Notes~4.9(1)]{Jones}, this action
$\pi_s$ gives an action of $L(\Gamma)$ on $A^2_{s-2}(\Hh)$, such that
\[
 \lambda_\gamma\cdot h=\pi_s(\gamma)h.
\]

With these preliminaries, we can describe the first part of the connection between cusp forms and random matrices.
Let $k$ be a positive even integer and let $f,g\in S_k(\Gamma)$ be cusp forms of weight $k$.
 In \cite[Proposition~7.1 and Remark~9.6]{Jones} it is proven that multiplication by $f$
 defines a bounded linear operator $\mathcal M_f\colon A^2_{s-2}(\Hh)\longrightarrow A^2_{k+s-2}(\Hh)$ given by
\begin{equation}\label{eq:Mf}
 \mathcal M_fh:=f\cdot h,
\end{equation}
and that the Toeplitz operators
\begin{equation}\label{eq:cusp-Toeplitz}
 T_{f,g}^{(s)}:=\mathcal M_f^*\mathcal M_g
\end{equation}
are elements of $M_s'$. In \cite[Notes~3.3 (2)]{Jones} it is shown that $M_s'$ carries a trace $\tau_s'$. We normalise
 it by $\tau_s'(I)=1$. By \cite[Theorem~9.3 and Remark~9.6]{Jones},
\begin{equation}\label{eq:traces}
 \tau_s'\bigl(T_{f,g}^{(s)}\bigr)
 =\frac{1}{\covol(\Gamma)}\langle g,f\rangle_{\mathrm{Pet}},
\end{equation}
where $\langle g,f\rangle_{\text{Pet}}$ is the Petersson inner product
\begin{equation}\label{Peterssinn}
 \langle g,f\rangle_{\text{Pet}}:=\int_{\Gamma\backslash\Hh}g(z)\overline{f(z)}y^k\frac{dx\,dy}{y^2}.
\end{equation}
For the modular group, we have, by \cite{Radulescu} and
 \cite[Theorems~9.2--9.3 and Remark~9.6]{Jones}, that the linear span of the Toeplitz operators
 \eqref{eq:cusp-Toeplitz} is dense in terms of
\begin{equation}
    \label{2norm} \|T\|_{2,\tau_s'}:=\bigl(\tau_s'(T^*T)\bigr)^{1/2}
\end{equation}
as $k,f,$ and $g$ vary. 

With this notation, we can also explain the term ``tracelike'' used in Definition~\ref{def:tracelike}. By \cite[Corollary~9.5]{Jones}, a
 unit vector $\xi\in A^2_{s-2}(\Hh)$ is tracelike precisely when
\begin{equation}
    \label{tau's} 
 \langle T\xi,\xi\rangle_s=\tau_s'(T)\quad\text{for all }T\in M_s'.
\end{equation}

We now specialise to the critical case $s=s_0=13$ (recall that we have already
 specialised to $\Gamma=\PSL$). Let $\xi$ be a unit tracelike vector in $A^2_{11}(\Hh)$. Since, by
 Proposition~\ref{prop:tracelike-frame}, its $\Gamma$-orbit is an orthonormal basis, the formula
\begin{equation}\label{eq:orbit-unitary}
 U_\xi\left(\sum_\gamma a_\gamma\delta_\gamma\right)
 :=\sum_\gamma a_\gamma\pi_{13}(\gamma)\xi
\end{equation}
defines a unitary operator $U_\xi\colon\ell^2(\Gamma)\longrightarrow A^2_{11}(\Hh)$ satisfying
 $U_\xi\lambda_\alpha=\pi_{13}(\alpha)U_\xi$. Consequently,
\begin{equation}\label{eq:unitary-algebras}
 U_\xi L(\Gamma)U_\xi^*=M_{13},\qquad \text{ and }\qquad
 U_\xi L(\Gamma)'U_\xi^*=M_{13}'.
\end{equation}
Thus, the unitary $U_\xi$ induces a $*$-isomorphism
 $M_{13}'\longrightarrow L(\Gamma)'$ given by $T\longmapsto U_\xi^*TU_\xi.$
Composing this with the anti-isomorphism
$L(\Gamma)'\longrightarrow L(\Gamma)$ of Proposition \ref{prop:group-standard}(ii)
we obtain the linear, positivity- and trace-preserving $*$-anti-isomorphism
\begin{equation}\label{eq:cusp-group-map}
 \mathfrak C_\xi\colon M_{13}'\longrightarrow L(\Gamma),
 \qquad \text{given by
 $\mathfrak C_\xi(T)
 :=J_0(U_\xi^*TU_\xi)^*J_0.$}
\end{equation}
Thus, in particular, $\mathfrak C_\xi$ is an isometry and we have\begin{equation}
    \label{frakC}
 \mathfrak C_\xi(AB)=\mathfrak C_\xi(B)\mathfrak C_\xi(A),\qquad
 \mathfrak C_\xi(A^*)=\mathfrak C_\xi(A)^*.
\end{equation}
By expanding $U_\xi^*A^*\xi$ in the basis $(\delta_\gamma)$ and applying $J_0$ we obtain \begin{equation}\label{eq:coordinate-coefficient}
 \ip{\mathfrak C_\xi(A)\delta_e}{\delta_\gamma}
 =\ip{A\pi(\gamma^{-1})\xi}{\xi}_{13}.
\end{equation}
Thus, the identity coefficient is $\tau'_{13}(A)$. 

The map $\mathfrak C_\xi$ is the identification of the commutant of the modular group action with the group von Neumann algebra of $\PSL$ which, as mentioned in the Introduction, comprises the second part of the connection between cusp forms and random matrices. Since, by the first part of this connection, the linear span of the Toeplitz operators $T_{f, g}^{(13)}$ is dense in the commutant in terms of the norm \eqref{2norm}, to determine explicitly the isometry $\mathfrak C_\xi$ it suffices to compute  $\mathfrak C_\xi(T_{f,g}^{(13)})$. We show that its group coefficients are explicit integrals: For $\gamma\in\Gamma$, we have
\begin{equation}\label{eq:cusp-group-coefficients}
 \begin{aligned}
 \bigl\langle\mathfrak C_\xi(T_{f,g}^{(13)})\delta_e,\delta_\gamma\bigr\rangle
 &=\bigl\langle T_{f,g}^{(13)}\pi_{13}(\gamma^{-1})\xi,\xi\bigr\rangle_{13}\\
 &=\int_{\Hh} g(z)\pi_{13}(\gamma^{-1})\xi(z)\,
 \overline{f(z)\xi(z)}\,y^{k+11}\,dx\,dy.
 \end{aligned}
\end{equation}
Therefore, an explicit tracelike $\xi$ converts every specified cusp form Toeplitz operator into an explicit convolution operator in $L(\Gamma)$.

\section{A construction of an explicit tracelike vector}\label{C}
For this section, let $\Gamma<\mathrm{PSL}_2(\R)$ be an arbitrary Fuchsian group of finite hyperbolic covolume, and fix the parameter $s$ with $1<s\leq s_0$. 

To construct an explicit tracelike vector, we start with the orbit of the reproducing kernel (Proposition \ref{prop:kernel-complete}) and define an operator
$
 \mathcal C\colon\ell^2(\Gamma)\longrightarrow A^2_{s-2}(\Hh).$ This operator has dense range but it is not necessarily an isometry. We thus replace it by the partial isometry $V$ in its ``polar decomposition'', and set
 $\xi_s:=V\delta_e$. The intertwining property of $V$ converts Parseval's identity in
 $\ell^2(\Gamma)$ into the defining identity of a tracelike vector. At $s=s_0$, $V$ is unitary and for $\Gamma=\PSL$ it is the operator $U_\xi$ in
 \eqref{eq:orbit-unitary}. Thus, in that case, the construction also makes the
 cusp form/group-von-Neumann-algebra correspondence explicit.

\subsection{The operator \texorpdfstring{$\mathcal C$}{C}}
In this section, we will use the projective action $\pi_s$ of Subsection~\ref{BergmanSpaces} rather than renormalising it to give an actual action. Hence let $\sigma_s$ be the multiplier and
$
 \sigma_s(e,\gamma)=\sigma_s(\gamma,e)=1.
$
On $\ell^2(\Gamma)$, use the twisted left regular operators
\[
 \lambda_\alpha^{\sigma_s}\delta_\beta
 :=\sigma_s(\alpha,\beta)\delta_{\alpha\beta}.
\]
They are unitary and satisfy
\[
 \lambda_\alpha^{\sigma_s}\lambda_\beta^{\sigma_s}
 =\sigma_s(\alpha,\beta)\lambda_{\alpha\beta}^{\sigma_s},
 \qquad
 L_{\sigma_s}(\Gamma):=\{\lambda_\gamma^{\sigma_s}:\gamma\in\Gamma\}''.
\]
This is the twisted group von Neumann algebra of \cite[Definition~2.1]{Jones}, with the multiplier convention fixed by \eqref{eq:projective-law}. When $\sigma_s=1$, these are the ordinary operators $\lambda_\alpha$ and the algebra $L(\Gamma)$ used above. The twisting cannot in general be removed for an arbitrary Fuchsian group \cite[Notes~4.9(2) and Appendix~A]{Jones}. The parameter $s$
 will be displayed when it is useful and suppressed from some notation once it is fixed.
 
Fix a $z_0 \in \Hh$ with a trivial stabiliser in $\Gamma$ (for $\Gamma=\PSL$, one may take $z_0=2i$) and set
\begin{equation}\label{eq:kappa-definition}
 \kappa_s(z):=K_{s,z_0}(z)
\end{equation}
where $K_{s,w}$ is the reproducing kernel vector introduced in Subsection~\ref{BergmanSpaces}.
For $\gamma\in\Gamma$, set
\begin{equation}\label{eq:kappa-orbit}
 \kappa_{s,\gamma}:=\pi_s(\gamma)\kappa_s.
\end{equation}
Unless otherwise indicated, $s$ will be fixed and we will abbreviate these to $\kappa$ and $\kappa_\gamma$.

\Needspace{5\baselineskip}
\begin{proposition}\label{prop:kappa-evaluation}
For $h\in A^2_{s-2}(\Hh)$ and $\gamma\in\Gamma$,
\begin{equation}\label{eq:kappa-evaluation}
 \bigl|\langle h,\kappa_\gamma\rangle_s\bigr|^2
 =\operatorname{Im}(z_0)^{-s}\,|h(\gamma z_0)|^2\operatorname{Im}(\gamma z_0)^s.
\end{equation}
\end{proposition}

\begin{proof}
By unitarity, \eqref{eq:reproducing}, and \eqref{eq:kernel-diagonal}, we have
\[
 \bigl|\langle h,\pi_s(\gamma)\kappa_s\rangle_s\bigr|
 =|(h|_s\gamma)(z_0)|.
\]
Using $\operatorname{Im}(\gamma z_0)
=\operatorname{Im}(z_0)/|c_\gamma z_0+d_\gamma|^2$,
we deduce \eqref{eq:kappa-evaluation}.
\end{proof}

Let $c_{00}(\Gamma)$ be the set of finitely supported vectors in $\ell^2(\Gamma)$. Define
$
 \mathcal C_0\colon c_{00}(\Gamma)\longrightarrow A^2_{s-2}(\Hh)$ by
 \begin{equation}\label{eq:C0}
 \mathcal C_0\mathbf a:=\sum_{\gamma\in\Gamma}a_\gamma\kappa_\gamma,
 \quad \text{for $\mathbf a=(a_\gamma)_{\gamma\in\Gamma}.$}
\end{equation}

\Needspace{5\baselineskip}
\begin{proposition}\label{prop:C-bounded}
The operator $\mathcal C_0$ extends uniquely to a bounded operator
\begin{equation}\label{eq:C-extension}
 \mathcal C\colon\ell^2(\Gamma)\longrightarrow A^2_{s-2}(\Hh).
\end{equation}
Moreover, the range is dense in $A^2_{s-2}(\Hh)$:
\begin{equation}\label{eq:C-dense-range}
 \overline{\ran\mathcal C}=A^2_{s-2}(\Hh).
\end{equation}
Its adjoint is the operator
\begin{equation}\label{eq:C-adjoint}
 (\mathcal C^*h)(\gamma)=\langle h,\kappa_\gamma\rangle_s
 \qquad(h\in A^2_{s-2}(\Hh),\ \gamma\in\Gamma).
\end{equation}
\end{proposition}

\begin{proof}
For every fixed hyperbolic radius $r>0$, the local estimate
\begin{equation}\label{eq:local-estimate}
 |h(w)|^2\operatorname{Im}(w)^s
 \ll_{r,s}\int_{D(w,r)}|h(z)|^2\operatorname{Im}(z)^{s}\,d \mu z
\end{equation}
holds for holomorphic $h$, where $D(w, r)$ is the hyperbolic ball of center $w$ and
radius $r$ (see \cite[(5.2)]{ImamogluOSullivan}).
 Since the stabiliser of $z_0$ is trivial and the action is properly discontinuous, one may
 choose $r>0$ so that the discs $D(\gamma z_0,r)$ are pairwise disjoint. Combining
 \eqref{eq:kappa-evaluation} and \eqref{eq:local-estimate}, and then summing over $\gamma$, gives the
 estimate
\begin{equation}\label{eq:Bessel}
 \sum_{\gamma\in\Gamma}\bigl|\langle h,\kappa_\gamma\rangle_s\bigr|^2
 \ll_{s,\Gamma,z_0}\|h\|_s^2.
\end{equation}
Consequently, for $\mathbf a\in c_{00}(\Gamma)$ and $h\in A^2_{s-2}(\Hh)$,
\[
 \bigl|\langle\mathcal C_0\mathbf a,h\rangle_s\bigr|
 \leq\|\mathbf a\|_{\ell^2(\Gamma)}
 \left(\sum_{\gamma\in\Gamma}\bigl|\langle\kappa_\gamma,h\rangle_s\bigr|^2\right)^{1/2}
 \ll_{s,\Gamma,z_0}\|\mathbf a\|_{\ell^2(\Gamma)}\|h\|_s.
\]
Taking $h=\mathcal C_0\mathbf a$ implies that $\mathcal C_0$ is bounded on $c_{00}(\Gamma)$ and therefore extends uniquely to
the bounded operator \eqref{eq:C-extension}.

If $h$ is orthogonal to every $\kappa_\gamma$, then \eqref{eq:kappa-evaluation} gives
 $h(\gamma z_0)=0$ for every $\gamma$. Theorem ~\ref{thm:Jones-zero} therefore gives $h=0$. Hence the closed span of the $\kappa_\gamma$ is all of
 $A^2_{s-2}(\Hh)$. Since their finite linear combinations lie in $\ran\mathcal C$, this proves
 \eqref{eq:C-dense-range}. Finally, \eqref{eq:C-adjoint} follows from the definition of the adjoint
 and the equality $\mathcal C\delta_\gamma=\kappa_\gamma$.
\end{proof}

\subsection{Polar decomposition}\label{partisom}
The operator $\mathcal C$ is compatible with the group action and has dense range, but it need not be an isometry.
 Its ``polar decomposition'' supplies the required normalisation. We will prove several completeness and operator compatibility properties some of which have appeared in a more general setting in \cite[Sections 4 and 7]{RvV}.

We first establish the existence of a square root operator.
\Needspace{5\baselineskip}
\begin{proposition}\label{prop:C-equivariant}
For every $\alpha\in\Gamma$,
\begin{equation}\label{eq:C-equivariant}
 \mathcal C\lambda_\alpha^{\sigma_s}=\pi_s(\alpha)\mathcal C.
\end{equation}
Consequently, $\mathcal C^*\mathcal C$ commutes with every $\lambda_\alpha^{\sigma_s}$. Its positive square root
\begin{equation}\label{eq:abs-C}
 |\mathcal C|:=(\mathcal C^*\mathcal C)^{1/2}
\end{equation}
also commutes with every $\lambda_\alpha^{\sigma_s}$.
\end{proposition}
\begin{proof}
For $\alpha,\beta\in\Gamma$, the projective representation law \eqref{eq:projective-law} implies
\[
 \mathcal C\lambda_\alpha^{\sigma_s}\delta_\beta
 =\sigma_s(\alpha,\beta)\mathcal C\delta_{\alpha\beta}
 =\sigma_s(\alpha,\beta)\kappa_{\alpha\beta}
 =\pi_s(\alpha)\kappa_\beta
 =\pi_s(\alpha)\mathcal C\delta_\beta.
\]
Linearity, density, and continuity yield \eqref{eq:C-equivariant}. Taking adjoints gives $(\lambda_\alpha^{\sigma_s})^*\mathcal C^*=\mathcal C^*\pi_s(\alpha)^*$. Multiplying on the left by $\lambda_\alpha^{\sigma_s}$ and on the right by $\pi_s(\alpha)$, and using unitarity, we deduce $\lambda_\alpha^{\sigma_s}\mathcal C^*=\mathcal C^*\pi_s(\alpha)$ and hence
$
 \mathcal C^*\mathcal C\lambda_\alpha^{\sigma_s}
 =\mathcal C^*\pi_s(\alpha)\mathcal C
 =\lambda_\alpha^{\sigma_s}\mathcal C^*\mathcal C.
$ Since $\langle \mathcal C^* \mathcal C x, x \rangle =
\langle \mathcal C x, \mathcal C x \rangle_s \ge 0$, 
the self-adjoint operator $\mathcal C^* \mathcal C$ has a unique square root $(\mathcal C^*\mathcal C)^{1/2}$ which is
positive, self-adjoint and commutes with $\lambda_{\gamma}^{\sigma_s}$ (see \cite[Theorem VI.9]{ReedSimon}).
\end{proof}
With this, we can construct the partial isometry associated with $\mathcal C$ and establish its key properties for our purposes.
\Needspace{5\baselineskip}
\begin{proposition}\label{prop:polar}
There is a partial isometry
\[
 V\colon\ell^2(\Gamma)\longrightarrow A^2_{s-2}(\Hh)
\]
such that
\begin{equation}\label{eq:polar}
 \mathcal C=V|\mathcal C|,\qquad \text{and \, \, $\ker V=\ker\mathcal C.$}
\end{equation}
It satisfies
\begin{equation}\label{eq:polar-projections}
 V^*V=P_{(\ker\mathcal C)^\perp},\qquad \text{and \, \, $VV^*=I_{A^2_{s-2}(\Hh)},$}
\end{equation}
and, for every $\alpha\in\Gamma$,
\begin{equation}\label{eq:V-equivariant}
 V\lambda_\alpha^{\sigma_s}=\pi_s(\alpha)V.
\end{equation}
\end{proposition}

\begin{proof}
Initially define $V$ on $\ran|\mathcal C|$ by
\begin{equation}\label{eq:V-initial}
 V(|\mathcal C|\mathbf a):=\mathcal C\mathbf a.
\end{equation}
This is well-defined and isometric because
\[
 \bigl\||\mathcal C|\mathbf a\bigr\|_{\ell^2(\Gamma)}^2
 =\langle\mathcal C^*\mathcal C\mathbf a,\mathbf a\rangle
 =\|\mathcal C\mathbf a\|_s^2.
\]
The same identity shows that
$
 \ker|\mathcal C|=\ker\mathcal C$, which in turn implies
\[ \overline{\ran|\mathcal C|}=(\ker |\mathcal C|)^\perp=(\ker\mathcal C)^\perp.
\]
Thus \eqref{eq:V-initial} extends to an isometry on $(\ker\mathcal C)^\perp$. Define it to be zero
 on $\ker\mathcal C$. By construction \eqref{eq:polar} holds. The final subspace $\overline{\ran(\mathcal C)}$ of $V$ equals $A_{s-2}^2$, by \eqref{eq:C-dense-range}. Since, by definition, the initial subspace of $V$ is $(\ker\mathcal C)^\perp$, we have, with the proposition preceding \cite[VI.10]{ReedSimon},
\begin{equation}\label{VV*}
 V^*V=P_{(\ker\mathcal C)^\perp},\qquad \text{and \, \, $VV^*=\operatorname{proj}_{\overline{\ran(\mathcal C)}}
 =I_{A^2_{s-2}(\Hh)}$}.
\end{equation}

Finally, the kernel of $\mathcal C$ is invariant under $\lambda_\alpha^{\sigma_s}$ by \eqref{eq:C-equivariant}.
 On $\ran|\mathcal C|$, with Proposition~\ref{prop:C-equivariant},
\[
 V\lambda_\alpha^{\sigma_s}|\mathcal C|\mathbf a
 =V|\mathcal C|\lambda_\alpha^{\sigma_s}\mathbf a=\mathcal C\lambda_\alpha^{\sigma_s}\mathbf a =\pi_s(\alpha)\mathcal C\mathbf a=\pi_s(\alpha)V|\mathcal C|\mathbf a.
\]
Continuity gives the identity on $(\ker\mathcal C)^\perp$ and both sides vanish on
 $\ker\mathcal C$. This shows \eqref{eq:V-equivariant}.
\end{proof}

\subsection{The explicit tracelike vector}\label{tracelikev}
We now apply the partial isometry to the basis vector $\delta_e$. Although $V$ depends
 on $s$, we suppressed this dependence in Section~\ref{partisom}. However, we will maintain it in the notation $\xi_s$ below.

\begin{theorem}\label{thm:tracelike}
For $1<s\leq s_0$, let
\begin{equation}\label{eq:xi-definition}
 \xi_s:=V\delta_e\in A^2_{s-2}(\Hh).
\end{equation}
Then we have
\begin{equation}\label{eq:Parseval}
 \sum_{\gamma\in\Gamma}\bigl|\langle h,\pi_s(\gamma)\xi_s\rangle_s\bigr|^2
 =\|h\|_s^2\qquad(h\in A^2_{s-2}(\Hh)).
\end{equation}
Moreover,
\begin{equation}\label{eq:xi-tracelike}
 \sum_{\gamma\in\Gamma}\bigl|\pi_s(\gamma)\xi_s(z)\bigr|^2
 =\frac{s-1}{4\pi}y^{-s}.
\end{equation}
Finally,
\begin{equation}\label{eq:xi-norm}
 \|\xi_s\|_s^2=\frac{s-1}{4\pi}\covol(\Gamma)=\frac{s-1}{s_0-1}.
\end{equation}
At $s=s_0$, $\{\pi_{s_0}(\gamma)\xi_{s_0}:\gamma\in\Gamma\}$ is an orthonormal basis of $A^2_{s_0-2}(\Hh)$.
\end{theorem}

\begin{proof}
By \eqref{eq:V-equivariant},
\[
 \pi_s(\gamma)\xi_s=\pi_s(\gamma)V\delta_e=V\lambda_\gamma^{\sigma_s}\delta_e=V\delta_\gamma.
\]
Hence, Parseval's identity in $\ell^2(\Gamma)$, combined with $VV^*=I_{A^2_{s-2}(\Hh)}$, gives
\[
 \sum_{\gamma\in\Gamma}\bigl|\langle h,\pi_s(\gamma)\xi_s\rangle_s\bigr|^2
 =\sum_{\gamma\in\Gamma}\bigl|\langle V^*h,\delta_\gamma\rangle_{\ell^2(\Gamma)}\bigr|^2=\|V^*h\|_{\ell^2(\Gamma)}^2=\langle VV^*h,h\rangle_s=\|h\|_s^2.
\]
This proves \eqref{eq:Parseval}. By Proposition \ref{prop:tracelike-frame}, this implies that $\xi_s$ is a tracelike vector with constant $C_{\xi}=(s-1)/(4 \pi)$ and hence, \eqref{eq:xi-tracelike} holds. With \eqref{eq:tracelike-norm} we deduce \eqref{eq:xi-norm}. If $s=s_0$, the norm is $1$ and the last part of
 Proposition~\ref{prop:tracelike-frame} shows that the $\Gamma$-orbit is an orthonormal basis.
\end{proof}

\begin{corollary}\label{cor:critical-unitary}
For $s=s_0$, $V_{s_0}$ is unitary and
\[
 V_{s_0}\delta_\gamma=\pi_{s_0}(\gamma)\xi_{s_0}
 \qquad(\gamma\in\Gamma).
\]
 In particular,
\[
 V_{s_0}L_{\sigma_{s_0}}(\Gamma)V_{s_0}^*=M_{s_0},\qquad
 V_{s_0}L_{\sigma_{s_0}}(\Gamma)'V_{s_0}^*=M_{s_0}'.
\]
For $\Gamma=\PSL$, with the normalisation of $\pi_s$ of Section~\ref{Parts12}, we have that $V_{13}$ is the unitary $U_{\xi_{13}}$ of \eqref{eq:orbit-unitary}, and  \eqref{eq:cusp-group-map} is obtained
 by substituting $U_\xi=V_{13}$ and $\xi=\xi_{13}$.
\end{corollary}

\begin{proof}
Theorem~\ref{thm:tracelike} shows that $V_{s_0}$ maps the standard orthonormal basis
 $(\delta_\gamma)_{\gamma\in\Gamma}$ onto the orthonormal basis
 $(\pi_{s_0}(\gamma)\xi_{s_0})_{\gamma\in\Gamma}$. The conjugation identities follow from \eqref{eq:V-equivariant}. For the modular group, we use from now on the normalised action of Section \ref{part1} also in the definitions of \(\mathcal C\) and \(V\). This leaves \(\xi_{13}\) unchanged and hence $V_{13}$ 
is exactly the unitary map appearing in \eqref{eq:orbit-unitary}.
\end{proof}

\subsection{Mbekhta's approximation of the polar decomposition}\label{sec:mbekhta} The polar decomposition of $\mathcal{C}$, and thus $\xi_s$, can be made numerically explicit. One of the direct ways to do that is based on Mbekhta's recursive approximation~\cite{Mbekhta}. For simplicity we outline it only in the case of $s=13$, $\mathcal H=A^2_{11}(\Hh)$, $\pi=\pi_{13}$, $\kappa=\kappa_{13}$, and $\xi=\xi_{13}$. Following~\cite[Corollary~3.1]{Mbekhta} we define
\begin{equation}\label{eq:polar-recursion}
 \begin{split}
 W_0&:=\mathcal C|\mathcal C|,\\
 W_{n+1}&:=W_n+\frac{1}{n+2}\bigl(W_0-\mathcal C\mathcal C^*W_n\bigr),
 \qquad n\ge0.
 \end{split}
\end{equation}
Then $W_n\to V_{13}$ strongly, and hence
\begin{equation}\label{eq:xi-recursion}
 \xi=\lim_{n\to\infty}W_n\delta_e
 \quad\text{in }\mathcal H.
\end{equation}
The result applies to the present setting by using the operator
\[
 \begin{pmatrix}0&0\\\mathcal C&0\end{pmatrix}
 \quad\text{on }\ell^2(\Gamma)\oplus\mathcal H,
\]
whose polar decomposition appears in the same lower-left block and zero in the other blocks.

For completeness, the recursion admits finite Gram approximations. Put $G=\mathcal C^*\mathcal C$ and define the polynomials
\[
 p_0(t)=1,\qquad
 p_{n+1}(t)=p_n(t)+\frac{1-tp_n(t)}{n+2}.
\]
Induction in \eqref{eq:polar-recursion} gives
$W_n=\mathcal C G^{1/2}p_n(G)$. Let $E_R$ be increasing finite sets that contain $e$ and exhaust $\Gamma$, let $Q_R$ be the corresponding coordinate projections, and let
\[
 (G_R)_{\alpha\beta}=\ip{\pi(\beta)\kappa}{\pi(\alpha)\kappa}_{13},
 \qquad\alpha,\beta\in E_R.
\]
These finite Gram matrices have entries determined by \eqref{eq:kernel}, \eqref{eq:kappa-definition}, and the fixed phases in $\pi$.

\Needspace{5\baselineskip}
\begin{proposition}[Finite approximation of the Mbekhta iterates]\label{prop:mbekhta-finite}
For each fixed $n\ge0$, the finite sums
\begin{equation}\label{eq:mbekhta-finite}
 \xi_{R,n}:=\sum_{\beta\in E_R}
 [G_R^{1/2}p_n(G_R)]_{\beta,e}\,\pi(\beta)\kappa
 \longrightarrow W_n\delta_e
\end{equation}
converge in $\mathcal H$ as $R\to\infty$. Consequently,
\[
 \lim_{n\to\infty}\lim_{R\to\infty}\xi_{R,n}=\xi.
\]
\end{proposition}
\begin{proof}
The operators $Q_RGQ_R$, extended by zero on the complementary subspace, converge strongly to $G$ and are bounded by $\|G\|$. For fixed $n$, the function $t^{1/2}p_n(t)$ is continuous on $[0,\|G\|]$. Uniform polynomial approximation and the continuous functional calculus (see \cite[Theorem VII.1]{ReedSimon}) give
\[
 (Q_RGQ_R)^{1/2}p_n(Q_RGQ_R)\delta_e
 \longrightarrow G^{1/2}p_n(G)\delta_e.
\]
The vector on the left belongs to $Q_R\ell^2(\Gamma)$ and is computed by the finite matrix $G_R^{1/2}p_n(G_R)$. Applying the bounded operator $\mathcal C$ shows \eqref{eq:mbekhta-finite}, and with \eqref{eq:xi-recursion} we obtain the second limit.
\end{proof}
Thus Mbekhta's recursion, together with finite Gram calculations at each fixed iteration, gives norm-convergent approximants to the tracelike vector.

\section{The cusp form operators and the random matrices}\label{sec:coefficients}
From now until the end of the paper we return to $\Gamma=\PSL$ and $s=s_0=13.$
\subsection{Normalisations and further notation}\label{sec:mass}
Let
\[
 X:=\Gamma\backslash\HH,\qquad
 d\mu(z):=\frac{dx\,dy}{y^2},\qquad d\nu:=\mu(X)^{-1}d\mu=\frac{3}{\pi}d\mu.
\]
Thus $\nu$ is hyperbolic area normalised to be a probability measure. We use the standard fundamental domain $\cF$. Start with
\[
 0\ne f\in S_k(\Gamma),\qquad 0\ne g\in S_\ell(\Gamma).
\]
The weights may differ, and neither form is required to be a Hecke eigenform. For a weight $k$ automorphic form $f$ (not necessarily holomorphic), write
\begin{equation}\label{eq:mass}
 N_f:= \norm{f}_{\text{Pet}}^2:=\langle f,f\rangle_{\text{Pet}}=\int_{X}|f(z)|^2y^k\,d\mu(z),\qquad
 v_f(z)=\frac{\pi}{3N_f}y^k|f(z)|^2.
\end{equation}
Let $P$ be the orthogonal projection of $L^2(\mathbb H,  y^{11} dx \, dy )$ onto its holomorphic subspace $A_{11}^2(\mathbb H)$ given by $(Ph)(z)=\langle h, K_{13, z}\rangle_{13}$. 
Comparing the two weighted inner products gives
\begin{equation}\label{eq:toeplitz}
 T_{f, f}^{(13)}=\mathcal M_f^*\mathcal M_f=P\mathcal M_{y^k|f|^2}|_{A_{11}^2(\mathbb H)}.
 \end{equation}
Set
 \begin{equation}\label{eq:toeplitz1} \That_f:=\frac{\pi}{3N_f}\mathcal M_f^*\mathcal M_f=P\mathcal M_{v_f}|_{A_{11}^2(\mathbb H)}
\end{equation}
From \eqref{eq:toeplitz} we easily deduce that $0 \le \That_f\le \|v_f\|_{\infty} I$. 

Use the single vector $\xi$ from Section~\ref{tracelikev} for every cusp form.  
Conjugation by the unitary $U_\xi=V_{13}$ identifies $M_{13}'$ with $L(\Gamma)'$ (see \eqref{eq:unitary-algebras}). The canonical trace on the latter is $R\mapsto\langle R\delta_e,\delta_e\rangle$, so it corresponds to the faithful, positive and normalised trace $\tau'_{13}$ of $M'_{13}$ given by \eqref{tau's}.
In particular, $\tau'_s(\That_f)=1$. 

For a pair of cusp forms $f,g$, our random matrices will model the correlation
\begin{equation}
    \label{correl}
C(f,g):=\tau'_s(\That_f\That_g)
\end{equation}
and the moments
\begin{equation}\label{eq:moments}
 m_r(f):=\tau'_s(\That_f^{\,r}).
\end{equation}


\subsection{The coefficient integrals}
In view of \eqref{eq:cusp-group-coefficients}, for our random matrix model we need
\begin{equation}\label{eq:af}
a_f(\gamma):=\frac{\pi}{3N_f}\int_{\HH}
 |f(z)|^2\,\pi(\gamma^{-1})\xi(z)\,\overline{\xi(z)}\,
 y^{k+13}\,d\mu(z).
\end{equation}
Replace $f,k,N_f$ by $g,\ell,N_g$ to obtain $a_g(\gamma)$. These integrals are absolutely convergent by boundedness of $y^k|f|^2$ and Cauchy--Schwarz in $A_{11}^2(\mathbb H)$. Equivalently,
$
 a_f(\gamma)=\ip{\That_f\pi(\gamma^{-1})\xi}{\xi}_{A_{11}^2(\mathbb H)}.
$
Set $$X_f:=\mathfrak C_\xi(\That_f).$$ Equations~\eqref{eq:coordinate-coefficient}, \eqref{eq:cusp-group-coefficients}, the adjoint identity, and Parseval in $L^2(L(\Gamma), \tau_\Gamma)$ imply
\begin{equation}\label{eq:coeffproperties}
 a_f(e)=1,\qquad a_f(\gamma^{-1})=\overline{a_f(\gamma)},
 \qquad \sum_{\gamma\in\Gamma}|a_f(\gamma)|^2=m_2(f).
\end{equation}
The conjugate in the middle identity is essential. In particular $a_f(S)$ is real and $a_f(U^2)=\overline{a_f(U)}$.

In group coordinates, the corresponding bounded left-convolution operator is
\begin{equation}\label{eq:Xf}
 X_f=\sum_{\gamma\in\Gamma}a_f(\gamma)\lambda_\gamma
 \quad\text{in }L^2(L(\Gamma), \tau_\Gamma),
\end{equation}
where $\lambda_\gamma\delta_\eta=\delta_{\gamma\eta}$. The map $\That_f\mapsto X_f$ preserves traces, adjoints, and positivity, but reverses products.

\subsection{The random matrices}\label{sec:matrices}
We next define the random matrixes we will need in the sequel.
Let $N$ be a positive integer divisible by $6$ and let $L$ be a nonnegative integer. We define $F_L$ to be the finite set of elements of $\PSL$ that can be expressed as the reduced product of at most $L$ elements from the set $\{S, U, U^2\}$ and the identity. Note that $F_L$ is closed under inversion.

Further, for each pair of cusp forms $f, g$ for $\PSL,$ fix constants
\begin{equation}\label{eq:Cbound}
 C_f\ge\norm{v_f}_\infty,\qquad C_g\ge\norm{v_g}_\infty.
\end{equation}

Next, let $U(N)$ be the unitary group. It is compact and has a probability measure $\mu_{\mathrm{Haar}}$ invariant under left and right multiplication by fixed unitary matrices. With this notation, let $(U(N),\mathcal B(U(N)),\mu_{\mathrm{Haar}})$ be the probability space, where $\mathcal B(U(N))$ is the Borel $\sigma$-algebra of $U(N)$ and consider the coordinate random variable $U_N(U)=U$.


For $\omega:=e^{2\pi i/3}$ and define
\begin{equation}\label{eq:generators}
 \begin{split}
 A_N&:=\diag(I_{N/2},-I_{N/2}),\\
 D_N&:=\diag(I_{N/3},\omega I_{N/3},\omega^2I_{N/3}),\\
 B_N&:=U_ND_NU_N^*.
 \end{split}
\end{equation}
Then
\begin{equation}\label{eq:relations}
 A_N^2=I_N,\qquad B_N^3=I_N.
\end{equation}
For $N=6n$, we have $A_N=A_6\otimes I_n$ and $D_N=D_6\otimes I_n$, where $\otimes$ denotes the Kronecker tensor product. Hence, for any polynomial $P$ in four variables, we have
\begin{equation}
    \label{Kron}
 P(A_N,D_N, A^*_N,D^*_N)=P(A_6,D_6, A^*_6,D^*_6)\otimes I_n.
 \end{equation}
Notice that $B_N$ is a matrix-valued random variable which is evaluated at the same outcomes $U_N \in U(N)$, for all $f$ and $g$. Their coefficients \eqref{eq:af} differ, but their random components are shared. This is what allows a mixed trace to retain the correlation between the forms. The matrices associated with the two forms are therefore generally dependent random variables.

\subsection{Substituting matrices for group elements}
\label{sec:substitute}
For elements $A,B$ of a unital algebra satisfying \eqref{eq:relations}, 
let $W_{\gamma}(A, B)$ denote the group homomorphism in $\gamma \in \PSL$ induced by the assignment $W_{S}(A, B)=A$ and $W_{U}(A, B)=B$. This is well-defined,
and we have $W_{\gamma^{-1}}(A, B)=W_{\gamma}(A, B)^{-1}.$ Further, \eqref{eq:relations} makes this a representation of $\Gamma$. Since $A_N$ and $B_N$ are unitary by \eqref{eq:generators}, this representation
is unitary. Thus,
$W_\gamma(A_N,B_N)^{-1}
= W_\gamma(A_N,B_N)^*.$
With this notation, and for $A=A_N$, $B=B_N$ define the finite sums
 \begin{equation}\label{eq:H}
 H_{f,N,L}:=\sum_{\gamma\in F_L}a_f(\gamma)W_{\gamma}(A_N, B_N),
\end{equation}
and likewise for $g$. These are continuous, hence measurable, matrix-valued functions of $U_N$, hence random matrices.
Using~\eqref{eq:coeffproperties}, we note that
\begin{align}\label{H*}
 H_{f,N,L}^*
 =\sum_{\gamma\in F_L}\overline{a_f(\gamma)}W_{\gamma^{-1}}(A_N, B_N) \nonumber
& =\sum_{\eta\in F_L}\overline{a_f(\eta^{-1})}W_{\eta}(A_N, B_N)
 \\
 &=\sum_{\eta\in F_L}a_f(\eta)W_{\eta}(A_N, B_N)=H_{f,N,L}.
\end{align}
This shows that its eigenvalues are real. The identical argument applies to $g$.

\subsection{Adjusting the eigenvalues and correcting the product reversal}
\label{sec:clip}
As will become clear in the proof of Theorem \ref{thm:joint}, the random matrix corresponding to $f$ needs to have eigenvalues in $[0,C_f]$. However, this is not the case for $H_{f,N,L}$. To address this, we use a technique of adjusting the eigenvalues that has previously been used, for example, in the proof of \cite[Lemma 1.4(6)]{Sh}. We first set
\[
 \chi_C(t)=\min\{C,\max\{0,t\}\}.
\]
If $H=Q\diag(t_1,\ldots,t_N)Q^*$ with $Q$ unitary, then we set
\[
 X_C(H):=Q\diag(\chi_C(t_1),\ldots,\chi_C(t_N))Q^*.
\]
In practice, this means that we replace negative eigenvalues by zero and eigenvalues greater than $C$ by $C$, while retaining the eigenvectors. With this notation, we can complete our construction by setting
\begin{equation}\label{eq:Z}
 Z_{f,N,L}:=\bigl[X_{C_f}(H_{f,N,L})\bigr]^T,
 \qquad Z_{g,N,L}=\bigl[X_{C_g}(H_{g,N,L})\bigr]^T.
\end{equation}
Transposition preserves positivity and eigenvalues. Hence, by our construction, the eigenvalues of $Z_{f,N,L}$ belong to $[0, C_f]$ and similarly for $Z_{g, N, L}.$
Transposition is continuous on Hermitian matrices, so the final matrices are measurable. 

We will need the following lemma.
\begin{lemma}\label{lem:clipping}
Let $(\mathcal M,\tau)$ be either
$(M_N(\mathbb C),\operatorname{tr}_N)$ or
$(L(\Gamma),\tau_\Gamma)$, and write
$\|Y\|_{2,\tau}:=\tau(Y^*Y)^{1/2}$.
If $P,X\in\mathcal M$, $P=P^*$ and $0\le X\le CI$, then
\[
 \|X_C(P)-X\|_{2,\tau}\le\|P-X\|_{2,\tau}.
\]
\end{lemma}
\begin{proof}
Set $D:=P-X_C(P)$. Using continuous
functional calculus, define
\[
 D_1=d_1(P), \, \, \text{and} \, \,  D_2=d_2(P),
 \qquad
\text{where $d_1(t):=\min\{t,0\}$ and $ d_2(t):=\max\{t-C,0\}$}.
\]
The scalar identities
$t-\chi_C(t)=d_1(t)+d_2(t), d_1(t)\chi_C(t)=0,
 d_2(t)\chi_C(t)=Cd_2(t)
$
and the inequalities $d_1(t)\le0$, $d_2(t)\ge0$ give
\[
 D=D_1+D_2,\qquad D_1\le0,\qquad D_2\ge0,
 \qquad D_1 X_C(P)=0,\qquad D_2 X_C(P) =CD_2.
\]
For positive operators $R,S\in\mathcal M$, $R^{1/2}SR^{1/2} \ge 0$, and hence
$\tau(RS)=\tau(R^{1/2}SR^{1/2})\ge0.$
Since $X\ge0$ and $CI-X\ge0$, it follows that
\[
 \tau(D_1(X_C(P)-X))=\tau((-D_1)X)\ge0,
 \qquad
 \tau(D_2(X_C(P)-X))=\tau(D_2(CI-X))\ge0.
\]
Hence $\tau(D(X_C(P)-X))\ge0$. As $P-X=D+(X_C(P)-X)$,
we deduce
\[
 \|P-X\|_{2,\tau}^2
 =\|D\|_{2,\tau}^2+\|X_C(P)-X\|_{2,\tau}^2
   +2\tau(D(X_C(P)-X))
 \ge \|X_C(P)-X\|_{2,\tau}^2.
\]
\end{proof}

\section{All joint moments and their modular-form interpretation}
\label{sec:meaning}
We now prove the main theorem which asserts that the explicit matrices from Section~\ref{sec:matrices} recover all joint moments of the cusp form operators. The result that will follow this theorem interprets the mixed second moment in terms of Petersson norms.

\subsection{The main theorem} Before we can state our main theorem, we introduce some terminology and notation and recall a key theorem we will need.

For an $N\times N$ matrix $A$, we consider the \emph{normalised} trace \[
\operatorname{tr}_N(A):=\frac1N\Tr(A), \quad \text{where
 $\Tr(A):=\sum_{j=1}^N A_{jj}$}.
\]
This normalisation ensures compatibility with the operator normalisation $\tau(I)=1$. 

We will say that $S(\omega)$ holds \emph{almost surely} in a probability space $(\Omega,\mathcal F,\mathbb P)$ if there
exists an event $E\in\mathcal F$ such that $\mathbb P(E)=1$
and $S(\omega)$ holds for every $\omega\in E$.

We now recall the following.  
\begin{theorem}[Theorem 1.4 of \cite{CM}]
\label{thm:collins-male}
Fix $p, q$. For each $N \in \mathbb N$ let
\(
 U_N=(U_1^{(N)},\ldots,U_p^{(N)})
\)
be a family of independent Haar-distributed $N\times N$
unitary random matrices, and let
$Y_N=(Y_1^{(N)},\ldots,Y_q^{(N)})$
be a family of $N\times N$ matrices independent of $U_N$. 

Let $(\mathcal A,\tau)$ be a norm-closed unital $*$-algebra of
bounded operators on a Hilbert space, equipped with
the operator norm and a faithful normalised trace $\tau$. Let $u_j \in \mathcal{A}$ ($j=1, \dots, p$) be free and satisfy $u_ju_j^*=u_j^*u_j=1$ and $\tau(u_j^k)=0$, for all non-zero $k \in \mathbb Z$.
Let $y=(y_1,\ldots,y_q)$ be a family
in $\mathcal{A}$ whose generated unital $*$-algebra is
free from that generated by $u$.

Assume that, almost surely, for every noncommutative
polynomial $Q$ in $2q$ variables,
\begin{equation}
    \label{hypo} \operatorname{tr}_N\!\bigl(Q(Y_N,Y_N^*)\bigr)
 \longrightarrow \tau\!\bigl(Q(y,y^*)\bigr),
 \qquad
 \bigl\|Q(Y_N,Y_N^*)\bigr\|
 \longrightarrow \bigl\|Q(y,y^*)\bigr\|.
\end{equation}
Then, almost surely, for every noncommutative polynomial
$P$ in $2p+2q$ variables,
\begin{equation}
    \label{concl}
 \operatorname{tr}_N\!\bigl(P(U_N,U_N^*,Y_N,Y_N^*)\bigr) \to
 \tau\!\bigl(P(u,u^*,y,y^*)\bigr), \quad
 \bigl\|P(U_N,U_N^*,Y_N,Y_N^*)\bigr\|
 \to
 \bigl\|P(u,u^*,y,y^*)\bigr\|.
\end{equation}
Here adjoints of tuples are taken componentwise and 
the matrix norms are operator norms.
\end{theorem}
Note that the ``almost surely'' qualification is not stated explicitly in the statement of \cite{CM}, but it is in the proof. 

We can now state and prove our main theorem.
\begin{theorem}\label{thm:joint}
Let $f_1,\ldots,f_m$ be a fixed finite collection of nonzero holomorphic cusp forms for $\Gamma$, possibly of different weights. For $6\mid N$ and $L\ge0$, let $Z_{f_j,N,L}$ be the positive matrix-valued random variables defined by \eqref{eq:af}, \eqref{eq:generators}, \eqref{eq:H}, and \eqref{eq:Z}.
Then, almost surely, for every $r\ge1$ and all $i_1,\ldots,i_r\in\{1,\ldots,m\}$,
\begin{equation}\label{eq:joint}
 \lim_{L\to\infty}\lim_{\substack{N\to\infty\\6\mid N}}
 \operatorname{tr}_N(Z_{f_{i_1},N,L}\cdots Z_{f_{i_r},N,L})
 =\tau'_{13}(\That_{f_{i_1}}\cdots\That_{f_{i_r}}).
\end{equation}
On the same event of probability one, for every
$f\in\{f_1,\ldots,f_m\}$, we further have
\begin{equation}\label{eq:unclipped}
 \lim_{L\to\infty}\lim_{\substack{N\to\infty\\6\mid N}}
 \operatorname{tr}_N(H_{f,N,L}^2)=m_2(f).
\end{equation}
\end{theorem}
\begin{proof}
First fix the word cutoff $L$ and let $N$ tend to infinity through multiples of $6$. We will specialise Theorem \ref{thm:collins-male} to the case we need it for. 

We first define the product probability space
\[
 (\Omega,\mathcal B_\Omega,\Prob)
 :=\prod_{n\ge1}\bigl(U(6n),\mathcal B(U(6n)),\mu_{\mathrm{Haar},6n}\bigr),
\]
and let $U_{6n}$ be the coordinate maps.

Next consider the reduced free product 
\[
 (\mathcal{A},\varphi)
 =\bigl(M_6(\C),\operatorname{tr}_6\bigr)
 *_{\mathrm{red}}\bigl(C(\mathbb T),m_{\mathbb T}\bigr),
 \qquad
 m_{\mathbb T}(h)=\frac1{2\pi}\int_0^{2\pi}h(e^{it})\,dt,
\]
of the tracial $C^*$-probability spaces $\bigl(M_6(\C),\operatorname{tr}_6\bigr)$ and $\bigl(C(\mathbb T),m_{\mathbb T}\bigr)$, where $C(\mathbb T)$ comprises the continuous complex-valued functions on the circle  $\mathbb T$ and $m_{\mathbb T}(h):=\frac1{2\pi}\int_0^{2\pi}h(e^{it})\,dt.$ See  \cite[Section 1]{VoiculescuSym} for the definition and details of the reduced free product, but, for our purposes we only need that it contains trace-preserving, isometric copies of both factors, that $\varphi$ is its faithful trace, and that the spaces are free from each other in the sense that
$\varphi(x_1x_2\cdots x_r)=0$
when \(x_j\) alternate between the two spaces and \(\varphi(x_j)=0\) ($j=1, \dots, r$).

Let $\iota_1:M_6(\mathbb C)\to \mathcal{A}$, $\iota_2\colon C(\mathbb T)\to \mathcal{A}$ be the natural embeddings of $M_6(\mathbb C)$ and $C(\mathbb T)$ to $\mathcal{A}$. Then set
\[
a_0=\iota_1(A_6),\qquad d_0=\iota_1(D_6),
\qquad u=\iota_2(z\mapsto z).
\]
It is clear that $u$ satisfies $uu^*=u^*u=1$ and $\varphi(u^k)=0$ ($k \in \mathbb Z \setminus \{0\}$). Since $u$ belongs to the image of $\iota_2$, it is free from $(a_0,d_0)$. 

We can therefore apply Theorem \ref{thm:collins-male} with $p=1, q=2$, $U_N$ the random matrices on $(\Omega,\mathcal B_\Omega,\Prob)$, $Y_1^{(N)}=A_N$, $Y_2^{(N)}=D_N$, $(\mathcal A, \varphi)$ as constructed above, $u_1=u$ and $y_1=a_0$, $y_2=d_0$.

Since $6| N$, \eqref{Kron} implies that the normalised trace and operator norm of $P(A_N, D_N, A^*_N, D^*_N)$ equal those of $P(A_6,D_6, A^*_6, D^*_6)$ and, hence, \eqref{hypo} holds trivially. Theorem \ref{thm:collins-male} gives \eqref{concl} and hence, on one event of probability $1$, for every noncommutative $*$-polynomial $Q$, we have
\begin{equation}
    \label{concl1}
 \operatorname{tr}_N Q(A_N,B_N)=\operatorname{tr}_N Q(A_N, U_ND_NU_N^*)
 \longrightarrow \varphi\bigl(Q(a_0,ud_0u^*)\bigr).
\end{equation}
To compute $\varphi\bigl(Q(a_0,ud_0u^*)\bigr)$,
we first show that the unital $*$-algebras generated by $a_0$ and $ud_0u^*$ are free. Every zero-trace element of the latter has the form
$uyu^*$ for some zero-trace $y$ in the algebra generated by $d_0$, since conjugation by $u$ preserves the trace.
Therefore, an alternating product of zero-trace elements from those two unital $*$-algebras becomes an alternating product of zero-trace elements from $\iota_1(M_6(\mathbb C))$ and the unital $*$-algebra generated by $u$. The factors from the latter are $u$ and $u^*$,
both of trace zero. The freeness of $u$ from $\iota_1(M_6(\mathbb C))$ shows that this product has trace $0$ and our unital $*$-algebras are free.

Moreover, $a_0^2=(ud_0u^*)^3=1$ and 
 $\varphi(a_0)=\varphi(ud_0u^*)=\varphi((ud_0u^*)^2)=0.$
Since every nonidentity $\gamma\in\Gamma=\PSL$ is a product that alternates between $S$ and either
$U$ or $U^2$, $W_\gamma(a_0,ud_0u^*)$ is an
alternating product of zero-trace elements from the
two free algebras above. Hence its trace is zero by freeness and hence
$\varphi\bigl(W_\gamma(a_0,ud_0u^*)\bigr)
 =\mathbf1_{\gamma=e}.$
The group trace likewise satisfies
$\tau_\Gamma(\lambda_\gamma)=\mathbf1_{\gamma=e}$. Thus, from \eqref{concl1}, for every fixed $\gamma\in\Gamma$, 
\begin{equation}\label{eq:freeness}
 \operatorname{tr}_N W_\gamma(A_N, B_N)\longrightarrow
 \tau_\Gamma(\lambda_\gamma)=\mathbf1_{\gamma=e}
 \quad\text{almost surely}.
\end{equation} 
Since $\Gamma$ is countable, the intersection of these probability-one events over all $\gamma$ is still an event $\Omega_0$ of probability one. Fix an outcome in $\Omega_0$ for the remainder of the proof.

The next step of the construction, involves the cutoff $L$. 
For each $f=f_j$ $(j=1, \dots m)$, set
\begin{equation}\label{eq:PY}
 P_{f,L}=\sum_{\gamma\in F_L}a_f(\gamma)\lambda_\gamma,
\quad  R_{f,L}=\sum_{\gamma\in F_L}|a_f(\gamma)|,
\quad  K_{f,N,L}=X_{C_f}(H_{f,N,L})
\end{equation}
and, by functional calculus, define $Y_{f,L}=X_{C_f}(P_{f,L})$.
As in \eqref{H*}, we see that $P_{f,L}$ is self-adjoint. Since every $\lambda_\gamma$ and every $W_\gamma(A_N,B_N)$ is unitary,
\begin{equation}\label{eq:fixedcutoffbounds}
 \|P_{f,L}\|\le R_{f,L},\qquad
 \|H_{f,N,L}\|\le R_{f,L}\quad\text{for every }N.
\end{equation}
Since $W_\gamma(A_N, B_N)$ is a group homomorphism in terms of $\gamma$, \eqref{eq:freeness} gives that on $\Omega_0$,
\begin{align}\label{firstlimit}
 \operatorname{tr}_N(H_{f_{i_1},N,L}\cdots H_{f_{i_r},N,L})&=\sum_{\gamma_1,\ldots,\gamma_r\in F_L}
 \left(\prod_{j=1}^r a_{f_{i_j}}(\gamma_j)\right)
 \operatorname{tr}_N W_{\gamma_1\cdots\gamma_r}(A_N,B_N) \nonumber\\
 &\longrightarrow
 \sum_{\substack{\gamma_1,\ldots,\gamma_r\in F_L\\
                  \gamma_1\cdots\gamma_r=e}}
 \prod_{j=1}^r a_{f_{i_j}}(\gamma_j)
 =\tau_\Gamma(P_{f_{i_1},L}\cdots P_{f_{i_r},L}).
\end{align}
The first implication of this is \eqref{eq:unclipped}. Indeed, let $r=2$, $f_{i_1}=f$ and $f_{i_2}=g$. Then, Cauchy-Schwarz for the inner product $\langle U, V\rangle_2:=\tau_\Gamma(V^*U)$ gives, for $X_f, X_g$ as in \eqref{eq:Xf},
\begin{equation}
    \label{r=2}
|\tau_\Gamma(P_{f,L}P_{g,L}-X_fX_g)|
\le
\|P_{f,L}-X_f\|_2\,\|P_{g,L}\|_2+\|X_f\|_2\,\|P_{g,L}-X_g\|_2
\end{equation}
Since, by construction, $P_{f, L} \to X_f$ and $P_{g, L} \to X_g$ in the norm $\| \cdot \|_2$, as $L \to \infty$, and $\|P_{g, L}\|_2 \le \|X_g\|_2$, we deduce, with \eqref{firstlimit},
\begin{equation}
    \label{r=2limit}
    \lim_{L\to\infty}\lim_{\substack{N\to\infty\\6\mid N}}
 \operatorname{tr}_{N}(H_{f, N, L} H_{g, N, L})=\tau_\Gamma(X_f X_g).
 \end{equation}
If now $g=f$, $\tau_\Gamma(X_f X_g)
=\sum_{\gamma}a_f(\gamma)a_f(\gamma^{-1})=\sum_{\gamma}|a_f(\gamma)|^2$ and hence \eqref{r=2limit} gives \eqref{eq:unclipped}. 

Unfortunately, we cannot extend the same argument to deduce \eqref{r=2limit} for $r>2$ because the analogue of \eqref{r=2} would involve factors $\|P_{g_1, L}P_{g_2, L} \dots\|_2$ instead of $\|P_{g, L}\|_2$. The problem is that the boundedness of the individual $\|P_{g_i, L}\|_2$ independently of $L$ does \emph{not} imply boundedness of $\|P_{g_1, L}P_{g_2, L} \dots\|_2$ independently of $L$. To address this difficulty, instead of $H_{f_j, N, L}$ we work with the $Z_{f_j, N, L}$, which were constructed in Subsection~\ref{sec:clip} specifically so that the necessary boundedness independently of $L$ can be ensured.

Let $r$ be any integer $\ge 1$ and let $\delta$ be an arbitrary element of $(0, 1)$. Since $\chi_{C_{f_j}}$ is continuous, we can choose real polynomials $p_{j,\delta}$ such that
\[
 \sup_{|t|\le R_{f_j,L}}
 |p_{j,\delta}(t)-\chi_{C_{f_j}}(t)|\le\delta, \qquad j \in \{1, \dots, m\}.
\]
By the isometry of the continuous functional calculus map
and \eqref{eq:fixedcutoffbounds}, \begin{equation}\label{unif0}
      \|p_{j,\delta}(H_{f_j,N,L})-K_{f_j,N,L}\| \le\delta,
\qquad \|p_{j,\delta}(P_{f_j,L})-Y_{f_j,L}\|\le\delta.
\end{equation}
We avoid the problem explained above because now $\|K_{f_j, N, L}\|
=
\|X_{C_{f_j}}(H_{f_j, N, L})\|
=\sup_{t\in\sigma(K_{f_j, N, L})}|\chi_{C_{f_j}}(t)|
\le C_{f_j}$, and similarly $\|Y_{f_j, L}\|\le C_{f_j}$. Thus, from \eqref{unif0}, their polynomial approximants have norm at most $C_{f_j}+\delta$. Write
$M=1+\max_{1\le j\le m}C_{f_j}$. Then, the identity
\begin{equation}
    \label{identity}
A_1\cdots A_r-K_1\cdots K_r
=\sum_{s=1}^r
A_1\cdots A_{s-1}(A_s-K_s)K_{s+1}\cdots K_r.
\end{equation}
for $A_i=p_{j_i,\delta}(H_{f_{j_i},N,L})$ and
$K_i=K_{f_{j_i},N,L}$,
combined with \eqref{eq:fixedcutoffbounds}, 
implies
\begin{equation}\label{unif}
      \|p_{j_1,\delta}(H_{f_{j_1},N,L})\dots p_{j_r,\delta}(H_{f_{j_r},N,L}) -K_{f_{j_1},N,L} \dots K_{f_{j_r},N,L}\| \le r\delta M^{r-1},
\end{equation} and the analogous inequality for products of $Y_{f_j, L}$, uniformly in $N$.

Both traces are normalised positive functionals, so $|\operatorname{tr}_N A|\le\|A\|$ and $|\tau_\Gamma(A)|\le\|A\|$. Applying \eqref{firstlimit} to the product of the polynomial approximants therefore gives
\[
  \limsup_{\substack{N\to\infty\\6\mid N}}
 \bigl|\operatorname{tr}_N(K_{f_{i_1},N,L}\cdots K_{f_{i_r},N,L})
 -\tau_\Gamma(Y_{f_{i_1},L}\cdots Y_{f_{i_r},L})\bigr|
 \le 2r\delta M^{r-1}.
 \]
Letting $\delta \to 0^+$ proves the fixed-cutoff limit
\begin{equation}\label{eq:clippedfixedlimit}
 \lim_{\substack{N\to\infty\\6\mid N}}
 \operatorname{tr}_N(K_{f_{i_1},N,L}\cdots K_{f_{i_r},N,L})
 =\tau_\Gamma(Y_{f_{i_1},L}\cdots Y_{f_{i_r},L}).
\end{equation}

We finally compute the limit as $L \to \infty$. 
Recall that $X_f=\mathfrak C_\xi(\That_f)$ is a bounded positive operator. Further, since $\widehat T_f=P\mathcal M_{v_f}|_{A^2_{11}(\Hh)}$ (see \eqref{eq:toeplitz}), we have
\[
\langle\widehat T_fh,h\rangle
=\langle\mathcal M_{v_f}h,h\rangle
=\int_{\mathbb H} v_f(z)|h(z)|^2 y^{11} \, d\mu (z) \le C_f\|h\|^2, 
\] and, since $\widehat T_f \to X_{f}$ preserves positivity we have $X_f\le C_fI$. With \eqref{eq:Xf} and \(\|X\|_2\) as in Proposition \ref{prop:group-standard}, we have, as $L \to \infty$,
\begin{equation}\label{eq:cutofftail}
 \|P_{f,L}-X_f\|_2^2
 =\sum_{\gamma\notin F_L}|a_f(\gamma)|^2
 \longrightarrow0,
\end{equation}
because of \eqref{eq:coeffproperties}. Lemma~\ref{lem:clipping} then gives, as $L \to \infty$,
\begin{equation}\label{eq:cutofftail2}
 \|Y_{f,L}-X_f\|_2\le\|P_{f,L}-X_f\|_2\longrightarrow0.
\end{equation}
Further, for bounded $A,B,D$, Cauchy--Schwarz gives
\[
 |\tau_\Gamma(ADB)|=|\tau_\Gamma(BA D)|
 \le\|BA\|_2\|D\|_2\le\|B\|\|A\|\|D\|_2.
\]
Applying this to \eqref{identity} with $A_j=Y_{f_{i_j},L}$ and $K_j=X_{f_{i_j}}$ implies, with \eqref{eq:cutofftail2},
\begin{equation}\label{eq:jointtailbound}
 \bigl|\tau_\Gamma(Y_{f_{i_1},L}\cdots Y_{f_{i_r},L})
       -\tau_\Gamma(X_{f_{i_1}}\cdots X_{f_{i_r}})\bigr| \le\sum_{j=1}^r
 \|Y_{f_{i_j},L}-X_{f_{i_j}}\|_2
 \prod_{h\ne j}C_{f_{i_h}}\longrightarrow0.
\end{equation}

We can now complete the proof of \eqref{eq:joint}. 
Since $Z_{f,N,L}=K_{f,N,L}^{T}$, for every $N,L$,
\[
 \operatorname{tr}_N(Z_{f_{i_1},N,L}\cdots Z_{f_{i_r},N,L})
 =\operatorname{tr}_N\bigl((K_{f_{i_r},N,L}\cdots K_{f_{i_1},N,L})^T\bigr)=\operatorname{tr}_N(K_{f_{i_r},N,L}\cdots K_{f_{i_1},N,L}).
\]
An application of \eqref{eq:clippedfixedlimit} and \eqref{eq:jointtailbound}, in this order, gives the limit $\tau_\Gamma(X_{f_{i_r}}\cdots X_{f_{i_1}})$. On the other hand, \eqref{frakC} gives
$X_{f_{i_r}}\cdots X_{f_{i_1}}
 =\mathfrak C_\xi(\That_{f_{i_1}}\cdots\That_{f_{i_r}})$
 and, since $\mathfrak C_\xi$ preserves traces, we deduce  $\tau_\Gamma(X_{f_{i_r}}\cdots X_{f_{i_1}})
 =\tau'_{13}(\That_{f_{i_1}}\cdots\That_{f_{i_r}})$, as required in
 \eqref{eq:joint}.
\end{proof}

{\it Remarks.} \\
1. Although the theorem is stated in terms of $\mathcal M_f^*\mathcal M_f$, it also gives joint $*$-moment models for the same-weight pair operators $\mathcal M_f^*\mathcal M_g$ appearing in Jones's \cite[Remark~9.6]{Jones}. Indeed, with $Q_h=\mathcal M_h^*\mathcal M_h$, an easy computation gives
\[
 \mathcal M_f^*\mathcal M_g
 =\frac14\bigl(Q_{f+g}-Q_{f-g}-iQ_{f+ig}+iQ_{f-ig}\bigr).
\]
For $h\ne0$, $Q_h=(3N_h/\pi)\widehat{T}_h$, while $Q_0=0$.
Enlarge the finite collection to include the nonzero forms occurring
above, and take the same linear combination of the matrices
$(3N_h/\pi)Z_{h,N,L}$, using the zero matrix when $h=0$.
Expansion of each joint $*$-word then reduces its convergence to
Theorem~\ref{thm:joint}. 

\noindent
2. With the notation \eqref{correl}, Theorem \ref{thm:joint} implies that, when $r=2$, the ordinary mixed random matrix trace has the deterministic limit
$C(f,g).$ This limit can be expressed as an inner product involving the masses $v_f$ and $v_g$. Indeed, \cite[Lemma 4.3]{RadulescuBerezin} asserts, in our notation, that
\begin{equation}\label{eq:Ckernel}
 C(f,g)=\ip{v_f}{B_{13}v_g}_{L^2(\nu)},
\end{equation}
where 
\begin{equation}\label{eq:Berezin}
 (B_{13}v_g)(z)=\frac{3}{\pi}
       \int_{\HH}q(z,w)^{13} v_g(w)\,d\mu(w),
 \qquad
 \text{for} \, \, q(z,w):=\frac{4\operatorname{Im}z\operatorname{Im}w}
                  {|z-\bar w|^2}.
\end{equation}

\subsection{A calculation entirely in terms of modular forms}
\label{sec:raised}
Define the Maass--Shimura raising operator and its iterates by 
\begin{equation}\label{eq:raising}
 \delta_k=\frac{1}{2\pi i}\frac{\partial}{\partial z}-\frac{k}{4\pi y},
 \qquad \delta_k^{(n)}=\delta_{k+2n-2}\cdots\delta_{k+2}\delta_k,
 \qquad \delta_k^{(0)}=1.
\end{equation}
As this operator increases the weight of an automorphic object of weight $k$ by $2$, $\delta_k^{(n)}$ maps weight $k$ objects to weight $k+2n$ objects. We also recall the Pochhammer symbol $(a)_n:=a(a+1)\cdots(a+n-1)$ and $(a)_0:=1$.

\begin{theorem}\label{thm:raised}
For nonzero $f\in S_k(\Gamma)$ and $g\in S_\ell(\Gamma)$, we have
\begin{equation}\label{eq:mixedseries}
 C(f,g)=\frac{4\pi}{(\ell+12)N_fN_g}
 \sum_{n=0}^\infty\frac{(4\pi)^{2n}}{n!(\ell+13)_n}
 \norm{f\delta_\ell^{(n)}g}_{\text{Pet}}^2.
\end{equation}
The series converges and has nonnegative terms. Its value is unchanged when $f,k$ and $g,\ell$ are interchanged. In particular, when $f=g$, this reduces to
\begin{equation}\label{eq:secondseries}
 m_2(f)=\frac{4\pi}{(k+12)N_f^2}
 \sum_{n=0}^\infty\frac{(4\pi)^{2n}}{n!(k+13)_n}
 \norm{f\delta_k^{(n)}f}_{\text{Pet}}^2.
\end{equation}
\end{theorem}

\begin{proof}
We first make a change of variables in the defining integral \eqref{eq:Berezin} of $B_{13}v_g$ appearing in the formula \eqref{eq:Ckernel}.

Set $u=(w-z)/(w-\bar z)$. If $z=x+iy$, we have
\[
 q(z,w)=1-|u|^2,\qquad
 \operatorname{Im}w
 =\frac{y(1-|u|^2)}{|1-u|^2},\qquad
 d\mu(w)=\frac{4 \operatorname{Im}(u)^2 d \mu u }{(1-|u|^2)^2}.
\] Hence $ (B_{13}v_g)(z)$ equals
\[\frac{4\,y^\ell}{N_g}
   \int_{\mathbb D}\left |
 g\!\left(\frac{z-\bar z u}{1-u}\right) \right |^2
       \frac{(1-|u|^2)^{\ell+11} \operatorname{Im}(u)^2}{ | 1-u|^{2\ell}}\, d \mu u=      \frac{4\,y^\ell}{N_g}
\left |\left |(1-u)^{-\ell}
 g\!\left(\frac{z-\bar z u}{1-u}\right) \right | \right |_{A^2_{\ell+11}(\mathbb D)}^2
\]
where $\left |\left |\cdot \right | \right |_{A^2_{\ell+11}(\mathbb D)}^2$ is the norm of the Bergman space in the unit disk $\mathbb D$ model, with measure $(1-|u|^2)^{\ell+11}\,du$. 
According to \cite[Proposition 17]{Za}, for a holomorphic function $g$ we have:
\begin{equation}\label{eq:elliptic}
 (1-u)^{-\ell}g\!\left(\frac{z-\bar z u}{1-u}\right)
 =\sum_{n\ge0}\frac{(-4\pi y)^n}{n!}\delta_\ell^{(n)}g(z)u^n,
 \qquad |u|<1, \, \, z=x+iy \in \mathbb H.
\end{equation}
We may therefore apply the
Parseval identity to the orthonormal basis of $A^2_{\ell+11}(\mathbb D)$ given in \cite[Proposition 4.7(1)]{Jones}
\[
 \begin{aligned}
 (B_{13}v_g)(z)
 &=\frac{4 \pi \,y^\ell}{N_g(\ell+12)}
   \sum_{n\ge0} \left | \sqrt{\frac{n!}{(\ell+13)_n}}\frac{(-4\pi y)^n}{n!}\delta_\ell^{(n)}g(z) \right |^2\\
 &=\frac{4\pi}{N_g(\ell+12)}
   \sum_{n\ge0}
   \frac{(4\pi)^{2n}y^{\ell+2n}
         |\delta_\ell^{(n)}g(z)|^2}
        {n!(\ell+13)_n}.
 \end{aligned}
\]
Substituting into the formula \eqref{eq:Ckernel}
\[
 C(f,g)=\ip{v_f}{B_{13}v_g}_{L^2(\nu)},
\]
and applying monotone convergence theorem, we obtain
\[
 C(f,g)=
 \frac{4\pi}{(\ell+12)N_fN_g}
 \sum_{n\ge0}\frac{(4\pi)^{2n}}{n!(\ell+13)_n}
 \norm{f\,\delta_\ell^{(n)}g}_{\text{Pet}}^{\,2}.
\]
The series is finite because
$0\le B_{13}v_g\le\norm{v_g}_\infty$ and
$\int_Xv_f\,d\nu=1$.
Finally, traciality gives
\[
 C(f,g)=\tau'_s(\That_f\That_g)
       =\tau'_s(\That_g\That_f)=C(g,f).
\]
Taking $g=f$ proves the asserted formula for $m_2(f)$.
\end{proof}

\end{document}